\documentclass[11pt]{amsart}
\usepackage{amssymb, amsfonts, amsmath, amsthm} 
\usepackage{tikz}
\usetikzlibrary{arrows}
\newtheorem{theorem}{Theorem}[section]
\newtheorem{lemma}[theorem]{Lemma}
\newtheorem{proposition}[theorem]{Proposition}
\newtheorem{corollary}[theorem]{Corollary}

\theoremstyle{definition}
\newtheorem{definition}[theorem]{Definition}
\newtheorem{notation}[theorem]{Notation}
\newtheorem{remark}[theorem]{Remark}
\newtheorem{defrem}[theorem]{Definition and Remark}
\newtheorem{remnotation}[theorem]{Remark and Notation}
\newtheorem{example}[theorem]{Example}
\numberwithin{equation}{section}

\newcommand{\cA}{ {\mathcal A} }
\newcommand{\cB}{ {\mathcal B} }
\newcommand{\cD}{ {\mathcal D} }

\newcommand{\cEplus}{ {\mathcal E}^{+} }

\newcommand{\cH}{ {\mathcal H} }

\newcommand{\cK}{ {\mathcal K} }
\newcommand{\cM}{ {\mathcal M} }

\newcommand{\cP}{ {\mathcal P} }
\newcommand{\cR}{ {\mathcal R} }
\newcommand{\cS}{ {\mathcal S} }
\newcommand{\cW}{ {\mathcal W} }
\newcommand{\cWplus}{ \cW^{+} }

\newcommand{\bB}{ {\mathbb B} }
\newcommand{\bC}{ {\mathbb C} }

\newcommand{\bN}{ {\mathbb N} }
\newcommand{\bR}{ {\mathbb R} }

\newcommand{\cf}{ \mbox{Cf} }
\newcommand{\Int}{ \mbox{Int} }

\newcommand{\Dcinfdiv}{ \Dc^{\mathrm{(inf-div)}} }
\newcommand{\Rcinfdiv}{ \cR_c^{\mathrm{(inf-div)}} }
\newcommand{\Pcinfdiv}{ \cP_c^{\mathrm{(inf-div)}} }

\newcommand{\ncpolcon}{ \bC \langle Z, Z^* \rangle }
\newcommand{\Dalg}{ \cD_{\mathrm{alg}} }
\newcommand{\Dc}{ \cD_c }
\newcommand{\cPplus}{ \cP^{+} }

\newcommand{\ncsercon}{ \bC_0 \langle \langle z,z^* \rangle \rangle }

\newcommand{\ncpolk}{ \bC \langle X_1, \ldots , X_k \rangle }

\newcommand{\ncserk}{\bC_0 \langle \langle x_1, \ldots , x_k
                       \rangle \rangle}

\title[Eta-diagonals and infinite divisibility for 
R-diagonals]{Eta-diagonal distributions and infinite \\
divisibility for R-diagonals}

\author[H. Bercovici]{Hari Bercovici}
\thanks{HB: supported in part by a grant from the 
National Science Foundation of the USA}
\address{Hari Bercovici: Department of Mathematics,
Indiana University, Bloomington, Indiana, USA.}
\email{bercovic@indiana.edu}
\author[A. Nica]{Alexandru Nica}
\thanks{AN: research supported by a Discovery Grant from 
NSERC, Canada.}
\address{Alexandru Nica: Department of Pure Mathematics, 
University of Waterloo, Ontario, Canada.}
\email{anica@uwaterloo.ca}
\author[M. Noyes]{Michael Noyes}
\address{Michael Noyes: Department of Mathematics, 
Bard High School Early College, New York, 
New York, USA.}
\email{mnoyes@bhsec.bard.edu}
\author[K. Szpojankowski]{Kamil Szpojankowski}
\address{Kamil Szpojankowski:
Department of Pure Mathematics,
University of Waterloo, Ontario, Canada \\
and Faculty of Mathematics and Information Science,
Warsaw University of Technology, Poland.}
\email{kszpojan@uwaterloo.ca, k.szpojankowski@mini.pw.edu.pl}

\begin{document}

\begin{abstract}
The class of $R$-diagonal $*$-distributions is fairly well understood 
in free probability.  In this class, we
 consider the concept of infinite divisibility
with respect to the operation $\boxplus$ of 
free additive convolution.  We exploit  the relation between free probability 
and the parallel (and simpler) world of Boolean probability. 
It is natural to introduce the concept of an
{\em $\eta$-diagonal distribution} that is the 
Boolean counterpart of an $R$-diagonal distribution.  
We establish a number of properties of 
$\eta$-diagonal distributions, then we examine the canonical 
bijection relating $\eta$-diagonal distributions to 
infinitely divisible $R$-diagonal ones.  The overall result is 
a parametrization of an arbitrary $\boxplus$-infinitely divisible 
$R$-diagonal distribution that can arise in a 
$C^{*}$-probability space, by a pair of 
compactly supported Borel probability measures on $[ 0, \infty )$.  
Among the applications of this parametrization, we prove
 that the set of $\boxplus$-infinitely divisible 
$R$-diagonal distributions is closed under the operation 
$\boxtimes$ of free multiplicative convolution.
\end{abstract}

\maketitle
\section{Introduction}

Free additive convolution $\boxplus$ is a binary operation on 
the set $\cP$ of Borel probability measures on $\bR$, reflecting
the addition operation for free selfadjoint elements in a 
noncommutative probability space.  The properties of this operation
parallel in many respects the ones of the usual convolution on
$\cP$,  for instance in the treatment of infinite divisibility.

One way to approach $\boxplus$-infinite divisibility is to
use a bijection constructed in \cite{BP1999} which 
relates free independence to another form of 
noncommutative independence, namely Boolean independence.
In this paper we focus on probability measures 
with compact support, so we view this bijection as a map
$\bB : \cP_c \to \Pcinfdiv$, where $\cP_c$ is the set of
probability measures with compact support on $\bR$, while 
$\Pcinfdiv$ consists of those measures $\mu \in \cP_c$ which are
$\boxplus$-infinitely divisible, that is, have the property
that for every $n \in \bN$, there exists $\mu_n \in \cP_c$
satisfying
\begin{equation}   \label{eqn:1z}
\underbrace{\mu_n \boxplus \cdots \boxplus \mu_n}_n = \mu.
\end{equation}
The bijection $\mathbb B$  connects the fundamental 
transforms of free and Boolean probability, the 
{\em $R$-transform} and respectively the {\em $\eta$-series}.
For $\mu \in \cP_c$, both of these transforms $R_{\mu} (z)$ and 
$\eta_{\mu} (z)$ are convergent power series.  The bijection 
$\bB$ is described by the equation
\begin{equation}   \label{eqn:1a}
R_{\bB ( \mu )} = \eta_{\mu}, \quad \mu \in \cP_c.
\end{equation}
More precisely, for every $\mu \in \cP_c$ there exists a uniquely determined measure
$\nu \in \Pcinfdiv$  such that 
$R_{\nu} = \eta_{\mu}$, and one defines $\bB ( \mu ) := \nu$. 

At the level of compactly supported distributions, the bijection 
$\mathbb B$ is precisely the
pa\-ra\-me\-tri\-za\-tion of $\boxplus$-infinitely divisible 
distributions provided in \cite{V1986}.  This was extended in 
\cite{BP1999} to the space $\cP$ of all Borel probability 
measures on $\bR$.  In a different direction, the bijection 
$\bB$ was extended in \cite{BN2008} to the space of joint 
distributions for $k$-tuples of selfadjoint elements in a 
$C^*$-probability space.  Our goal in this paper is to use a 
multivariate version of the bijection $\bB$ in order to study 
$\boxplus$-infinitely divisibile $R$-diagonal 
distributions, a significant class of $*$-distributions 
considered in free probability.

To explain our results, we introduce some notation. We let  $\cD_c (1,*)$ stand for the collection of all $*$-distributions
of (generally, not selfadjoint) elements in a (generally not tracial) $C^*$-probability space.  There is
a natural operation $\boxplus$ on $\cD_c (1,*)$ which 
corresponds to the addition $a+b$ of two variables $a,b$ 
in the same space such that $\{a,a^*\}$ is free from $\{b,b^*\}$. 
Infinite divisibility in  $\cD_c (1,*)$ is defined 
as in (\ref{eqn:1z}), and we denote by
$\cD_c (1,*)^{\mathrm{(inf-div)}}$ the collection of 
$\boxplus$-infinitely divisible elements
of $\cD_c (1,*)$.  The notions of $R$-transform and $\eta$-series 
also have natural extensions
to the context of $*$-distributions. 

The results of \cite{BN2008}, specialized to two selfadjoint variables, can be applied to $\cD_c (1,*)$ after a simple change of
 coordinates.  There is again a bijection 
$\bB_{(1,*)} : \cD_c (1,*) \to 
\cD_c (1,*)^{\mathrm{(inf-div)}}$ defined by the requirement that
\begin{equation}   \label{eqn:1b}
R_{\bB_{ (1,*)} ( \mu )} = \eta_{\mu}, \quad \mu \in \cD_c (1,*).
\end{equation}
This is analogous to the condition (\ref{eqn:1a}) satisfied by the original bijection $\mathbb B$, but proving the existence of $\bB_{(1,*)}$
is more than a trivial extension of the proof for $\bB$, and 
requires a mixture of combinatorial and analytic methods.

We turn now to $R$-diagonal $*$-distributions, which can be 
succinctly described as the distributions in $\Dc (1,*)$ 
that are invariant under multiplication by a free Haar unitary
(see  \cite[Theorem 15.10, p.~244]{NS2006}).  For our purposes,
 it is more useful to consider the original definition \cite{NS1997} 
of $R$-diagonal distributions which asks that the $R$-transform of 
the distribution be in some sense `diagonal'  
\cite[Definition 15.3, p.~241]{NS2006}.  From this point of view, 
it is clear how to define the Boolean counterpart of $R$-diagonality:
we simply say that a $*$-distribution is 
{\em $\eta$-diagonal} if its $\eta$-series is diagonal.
The map  $\bB_{ (1,*) }$ defined by (\ref{eqn:1b}) 
will then give a bijection between the set of all
$\eta$-diagonal distributions in $\cD_c (1,*)$ and the set of 
$R$-diagonal distributions in $\cD_c (1,*)$ which are 
$\boxplus$-infinitely divisible.

The above discussion shows that there is some interest in studying
$\eta$-diagonal distributions.  In this paper we point out 
a few general algebraic and combinatorial properties of such a 
distribution $\mu$, which actually hold for $\mu$ in 
a larger, purely algebraic space $\Dalg (1,*)$.  
The property of a distribution 
$\mu \in \Dalg (1,*)$ of being $\eta$-diagonal has an elegant 
description phrased directly in terms of 
the $*$-moments of $\mu$.  This result (Theorem \ref{thm:28})
is reminiscent of (but simpler than) the description
\cite[Theorem  1.2.1]{NSS2001} of
$R$-diagonal variables in terms of their $*$-moments.  
The $\eta$-diagonal 
distributions also have other algebraic and combinatorial
properties that are analogous to known properties of 
$R$-diagonal distributions.
In particular, if $a$ is an $\eta$-diagonal element in a 
$*$-probability space $( \cA , \varphi )$ (which means, by 
definition, that $a$ has $\eta$-diagonal $*$-distribution 
with respect to $\varphi$) then it follows that $aa^{*}$
and $a^{*}a$ are Boolean independent elements of $\cA$, and 
that the coefficients of the $\eta$-series of $aa^{*}$ and 
$a^{*}a$ are read from the so-called \emph{determining sequences} 
for the $*$-distribution of $a$.  For details on the terms used
above and for a discussion of why this is
indeed analogous to known facts about $R$-diagonals, see Remark \ref{rem:34} below.

In the case in which the $\eta$-diagonal distribution $\mu$ is
in $\cD_c (1,*)$, we point out a natural 
parametrization for $\mu$, given by a pair of compactly 
supported Borel probability measures on $[ 0, \infty )$.  That is, 
we establish a canonical bijection
\begin{equation}   \label{eqn:1c}
\{ \mu \in \cD_c (1,*) :\mu \mbox{ is $\eta$-diagonal} \}
\ni \mu \leftrightarrow 
( \sigma_1 , \sigma_2 ) \in \cPplus_c \times \cPplus_c ,
\end{equation}
where $\cPplus_c := \{ \sigma \in \cP_c :
\sigma \bigl( \, [0, \infty ) \, \bigr) = 1 \}$.
Without going into details, we mention that all the 
$*$-distributions appearing in this paper are defined as linear
functionals on the algebra $\bC \langle Z, Z^{*} \rangle$ of complex
polynomials in the non-commuting indeterminates $Z$ and $Z^{*}$,
and that the correspondence 
$\mu \leftrightarrow ( \sigma_1, \sigma_2 )$ from 
(\ref{eqn:1c}) amounts to the equalities
\begin{equation}   \label{eqn:1d}
\mu (( ZZ^{*})^n) = \int_0^{\infty} t^n \, d \sigma_1 (t)
\mbox{ and }
\mu (( Z^{*}Z)^n) = \int_0^{\infty} t^n \, d \sigma_2 (t),
\quad n \in \bN .
\end{equation}
In other words, the probability measures $\sigma_1$ and $\sigma_2$ 
which parametrize $\mu$ in
(\ref{eqn:1c}) are simply the distributions of $ZZ^{*}$ and of
$Z^{*}Z$ with respect to the functional $\mu$.  The
relevant point here is that for any given 
$\sigma_1, \sigma_2 \in \cPplus_c$ 
there exists a unique $\eta$-diagonal distribution 
$\mu \in \cD_c (1,*)$ such that (\ref{eqn:1d}) holds.

When the bijection $\bB_{(1,*)}$ is applied to 
$\{ \mu \in \cD_c (1,*) : \mu \mbox{ is $\eta$-diagonal} \}$
in (\ref{eqn:1c}), we obtain a bijection
\begin{equation}   \label{eqn:1f}
\Bigl\{ \nu \in \cD_c (1,*) :
\begin{array}{ll}
 \nu \mbox{ is $R$-diagonal and}  \\
\mbox{$\boxplus$-infinitely divisible}  
\end{array}  \Bigr\}
\ni \nu \leftrightarrow 
( \sigma_1 , \sigma_2 ) \in \cPplus_c \times \cPplus_c .
\end{equation}
Thus, we have a parametrization of a general 
$\boxplus$-infinitely divisible $R$-diagonal distribution 
by a pair of probability measures from $\cPplus_c$.  
Analogously to (\ref{eqn:1d}), it is possible to write 
explicitly the relation
connecting $\sigma_1, \sigma_2$ to the distributions of the 
elements $ZZ^{*}$ and $Z^{*}Z$ in the noncommutative probability
space $( \bC \langle Z, Z^{*} \rangle , \nu )$.  Theorem \ref{thm:62} below
realizes this parametrization by providing precise formulas 
for the $R$-transforms of $ZZ^{*}$ and of $Z^{*}Z$.

An important subclass of $R$-diagonal distributions are 
those which satisfy the KMS condition for some parameter 
$t \in ( 0, \infty )$.  This is a generalization of the trace
condition, where the latter corresponds to the special case 
$t=1$ (see the review in Section 3 below).  For an $R$-diagonal
distribution $\nu$ which satisfies KMS with parameter $t$, one 
can process further the result of Theorem \ref{thm:62} in 
order to obtain explicit formulas for the distributions of 
$ZZ^{*}$ and of $Z^{*}Z$, in terms of the probability 
measures $\sigma_1$ and $\sigma_2$ which parametrize $\nu$. 
These formulas invoke some commonly used elements of free 
harmonic analysis on $\cPplus_c$, and are given in  
Proposition \ref{prop:66}.

As an application of the parametrization from (\ref{eqn:1f}),
we prove that the set of $\boxplus$-infinitely divisible 
$R$-diagonal $*$-distributions is closed under the operation
$\boxtimes$ of multiplicative convolution.

\vspace{10pt}

In addition to the present 
introduction, the paper contains 6 sections.  
Section 2 introduces $\eta$-diagonal $*$-distributions 
and discusses some of their algebraic properties.  Section 3 is 
devoted to a review of $R$-diagonal $*$-distributions, with 
emphasis on facts that are needed in the present paper.  In 
Section 4 we verify that the bijection $\bB_{(1,*)}$ does indeed 
work on $\cD_c (1,*)$ in the way described in (\ref{eqn:1b}).
Section 5 presents the operator model for 
$\eta$-diagonals that was announced in (\ref{eqn:1c}).
Section 6 contains our results concerning the parametrization 
of infinitely divisible $R$-diagonal distributions, and 
a discussion of the KMS example.  Finally, Section 7 discusses
the application to free multiplicative convolution.

\section{$\eta$-series and $\eta$-diagonal $*$-distributions}

\begin{notation}   \label{def:21}

(1) We denote by $\cWplus$ the set 
$\bigsqcup_{n=1}^{\infty} \{ 1,* \}^n$ consisting of all
non-empty words over the two-letter alphabet $\{1,* \}$.  This is a 
semigroup (without unit) under the natural operation of
concatenation.  We denote by $|w|$ the number of letters in a word
$w \in \cWplus$.

(2) The algebra of complex
polynomials in two non-commuting variables $Z$ and $Z^*$ is denoted, 
as usual, by $\mathbb{C}\langle Z, Z^*\rangle$.  For
every word $w = ( \ell_1, \ldots , \ell_n ) \in \cWplus$ we write
\[
Z^{w} = Z^{\ell_1} \cdots  Z^{\ell_n} \in \ncpolcon .
\]
The set $\{ 1 \} \cup \{Z^{w} : w \in \cWplus \}$ is a
basis of $\ncpolcon$ as a complex vector space.

(3) An \emph{algebraic $*$-distribution} is a linear functional 
$\mu: \ncpolcon \to \bC$ such that  $\mu (1) = 1$.  (At this stage
 we do not require $\mu$ to have any additional
properties.)  The values
of $\mu$ on monomials $Z^w$ (with $w \in \cWplus$) will be 
referred to as {\em $*$-moments} of $\mu$.

(4) The collection of all algebraic $*$-distributions from (3)
is denoted $\Dalg (1,*)$. 
\end{notation}

$ $

\begin{notation}   \label{def:22}
 ({\em Series and their coefficients.})
(1) The algebra of formal power series in two non-commuting
indeterminates $z$ and $z^{*}$ is denoted, as usual, by 
$\bC \langle \langle z,z^{*} \rangle \rangle$. The collection
$\ncsercon \subset
\bC \langle \langle z,z^{*} \rangle \rangle$ 
of power series with vanishing constant coefficient is a two-sided 
ideal in $\bC \langle \langle z,z^{*} \rangle \rangle$.
An arbitrary element $f\in \ncsercon$ is of the form
\begin{equation}   \label{eqn:22a}
f(z,z^*) = \sum_{n=1}^{\infty}
\sum_{{\ell_1,\ldots, \ell_n\in\{1,*\}}} 
\alpha_{( \ell_1,\ldots, \ell_n)}
z^{\ell_1}\cdots z^{\ell_n}
= \sum_{w \in \cWplus} \alpha_w z^w,
\end{equation}
where the coefficients $\alpha_w$ are complex 
numbers, and for $w = ( \ell_1, \ldots , \ell_n ) \in \cWplus$ 
we use the notation 
$z^w=z^{\ell_1} \cdots z^{\ell_n}$.

(2) Given $w \in \cWplus$, we denote by 
$\cf_w : \ncsercon \to \bC$ the linear functional which extracts 
the coefficient of $z^w$ from a series 
$f \in \ncsercon$.  That is, if $f$ is given by 
\eqref{eqn:22a}, we have $\cf_w (f) = \alpha_w$, $w \in \cWplus$.

(3) Given a positive integer $n$, a word 
$w = ( \ell_1, \ldots , \ell_n ) \in\mathcal{W}^+$, and a 
partition $\pi$ of $\{ 1, \ldots , n \}$, we define a  functional (non-linear
unless $\pi$ consists of only one block) $\cf_{w ; \pi} : \ncsercon \to \bC$,
as follows.  For every block $B = \{ b_1, \ldots , b_m \}$ of $\pi$,
where $1 \leq b_1 < \cdots < b_m \leq n$, we set
\[
w \vert B = ( \ell_1, \ldots , \ell_n) \vert B  
:=  ( \ell_{b_1}, \ldots , \ell_{b_m}) \in \{ 1, * \}^m.
\]
Then we define
\begin{equation}  \label{eqn:22b}
\cf_{w ; \pi} (f)  := \
\prod_{B \in \pi}  \cf_{w|B} (f),\quad f \in \ncsercon .
\end{equation}
\item[][Suppose, for instance, that $n=5$, 
$\pi = \{ \{ 1,4,5 \} , \{ 2,3 \} \}$, and
$w = ( \ell_1, \ldots , \ell_5)$.
Then $\cf_{w ; \pi } (f) =$
$\cf_{( \ell_1, \ell_4, \ell_5)} (f) \cdot 
\cf_{( \ell_2, \ell_3)} (f)$, $f \in \ncsercon$.]
\end{notation}

\begin{defrem}   \label{def:23}
({\em Moment series, $\eta$-series.})
Fix $\mu\in \Dalg (1,*)$.
(1) The {\em moment series} of $\mu$ is defined as
$M_{\mu} := \sum_{w \in \cWplus} \mu (Z^w) z^w \in \ncsercon$.

(2) The {\em $\eta$-series} of $\mu$ is defined as 
\begin{equation}   \label{eqn:23a}
\eta_{\mu} := M_{\mu} ( 1+ M_{\mu} )^{-1} 
= ( 1+ M_{\mu} )^{-1} M_{\mu} \in \ncsercon ,
\end{equation}
where all the algebraic operations are performed in the 
algebra $\bC \langle \langle z,z^{*} \rangle \rangle$.

(3) It is immediate from  (\ref{eqn:23a}) that the 
series $M_{\mu}$ can be retrieved from $\eta_{\mu}$ by the formula
\begin{equation}   \label{eqn:23b}
M_{\mu} = \eta_{\mu} ( 1- \eta_{\mu} )^{-1}
= ( 1- \eta_{\mu} )^{-1} \eta_{\mu}.
\end{equation}

(4) The right-hand side of (\ref{eqn:23b}) can be
written as a geometric series $\sum_{n=1}^{\infty} \eta_{\mu}^n$ (which converges 
in the sense that  the series
$\sum_{n=1}^\infty {\mathrm{Cf}}_w (\eta_\mu ^n)$ contains only finitely many 
non-zero terms for every $w \in \cWplus$).  This leads to an explicit formula for 
the coefficients of $M_{\mu}$ in terms of those of $\eta_{\mu}$, namely
\begin{equation}  \label{eqn:23c}
\cf_w (M_{\mu}) = \sum_{\pi \in \Int (n)}
\cf_{w ; \pi} ( \eta_{\mu} ),\quad w \in \cWplus \mbox{ with } |w| = n.
\end{equation}
Here $\Int (n)$ denotes the set of {\em interval partitions} of 
$\{ 1, \ldots , n \}$, that is, partitions  which have the property that every block 
$B$ of $\pi$ is  of the form  $\{a,a+1,\dots,b\}$ for some $a \leq b$ in 
$\{ 1, \ldots , n \}$.  An analogous argument converts (\ref{eqn:23a}) into the formula
\begin{equation}  \label{eqn:23d}
\cf_w ( \eta_{\mu} ) = \sum_{\pi \in \Int (n)}
(-1)^{1+ | \pi |} \cf_{w; \pi} ( M_{\mu} ),\quad w \in \cWplus \mbox{ with } |w| = n,
\end{equation} 
where $| \pi |$ denotes
the number of blocks of the partition $\pi$.
\end{defrem}

\begin{remark}   \label{rem:24}
It is clear that the map 
$\Dalg (1,*) \ni \mu \mapsto M_{\mu} \in \ncsercon$
is bijective.  Equations 
(\ref{eqn:23a}) and (\ref{eqn:23b}) show that the map 
$\Dalg (1,*) \ni \mu \mapsto \eta_{\mu} \in \ncsercon$
is a bijection as well.  In other words, we can define a distribution 
$\mu \in \Dalg (1,*)$ by specifying its $\eta$-series.
\end{remark}

\begin{definition}  \label{def:25}
A word $w \in \cWplus$ is said to be {\em alternating}
when it is of the form
$$w =( \underbrace{ 1,*,1,*, \ldots , 1,*}_{2m} )
   = (1,*)^m  \mbox{ or }
w =( \underbrace{ *,1,*,1, \ldots , *,1  }_{2m} )
   = (*,1)^m ,$$ 
for some positive integer $m$.
In the first case $w$ is said to be {\em of type $(1,*)$},
and in the second case $w$ is said to be
{\em of type $(*,1)$}. In these formulas, powers are taken relative to
concatenation.  Note  in particular that  alternating words 
have positive, even length.
\end{definition}

\begin{definition}  \label{def:26}
(1) A distribution $\mu\in\Dalg (1,*)$ 
is said to be {\em $\eta$-diagonal}
if $\cf_w ( \eta_{\mu } ) = 0$  for every word 
$w \in \cWplus$ which is not alternating.

(2) If $\mu$ is $\eta$-diagonal, its $\eta$-series is thus 
of the form
\[
\eta_{\mu} (z,z^{*}) = 
\sum_{n=1}^{\infty} \alpha_n (zz^{*})^n
+ \sum_{n=1}^{\infty} \beta_n (z^{*}z)^n,
\]
with $\alpha_n, \beta_n \in \bC$ for $n \in \bN$.  The
sequences $( \alpha_n )_{n=1}^{\infty}$ and 
$( \beta_n )_{n=1}^{\infty}$ will be called the 
{\em determining sequences} of the $\eta$-diagonal 
distribution $\mu$.
\end{definition}

The main goal of the present section is to reveal an 
equivalent characterization of $\eta$-diagonal
distributions, which is phrased directly in terms of 
$*$-moments.  For this purpose, we require one more 
concept related to words in $\cWplus$.

\begin{defrem}  \label{def:27}
(1) A word $w=( \ell_1, \ldots , \ell_n ) \in \cWplus$ is said 
to be {\em mixed-alternating} if $n=2m$ is even and if the
letters of $w$ are such that
$\ell_{2k-1 }\neq \ell_{2k}$, for $k=1,2,\dots,m$.
Equivalently, $w$ is mixed-alternating when it belongs to the concatenation 
subsemigroup of $\cWplus$ generated by the words $(1,*)$ and $(*,1)$.

(2) By grouping factors, it is
easily seen that every mixed-alternating word $w$ can be written in
a unique way as a concatenation of alternating words such that consecutive
words are of different types.  Indeed, if we write such a word as
\begin{equation}   \label{eqn:27b}
w = w_1 w_2 \cdots w_p 
\end{equation}
where $p \geq 1$, each $w_i$ is alternating,
and $w_i$ is not of the same type as
$w_{i+1}$, $i=1,2,\dots,p-1$,
then the boundaries between the words $w_1, \ldots , w_p$ can be 
retrieved at the places where $w$ has two consecutive 
identical letters.

\noindent
[For  example, $w = (1,*,1,*,1,*,*,1,*,1,1,*,*,1,*,1)$
is mixed-alternating and its canonical factorization 
(\ref{eqn:27b}) is $w_1 w_2 w_3 w_4$ with $w_1 = (1,*)^3$, 
$w_2 = (*,1)^2$, $w_3 = (1,*)$, $w_4 = (*,1)^2$.]
\end{defrem}

\begin{theorem}  \label{thm:28}
For every distribution $\mu \in \Dalg (1,*)$, the statements
\emph{(a)} and \emph{(b)} are equivalent.
\begin{enumerate}
\item[\rm(a)] $\mu$ is $\eta$-diagonal.

\item[\rm(b)] $\mu$ satisfies the following 
conditions\footnote{The acronym $\eta$DM is meant to suggest 
$\eta$-Diagonality-in-Moments.}\emph{:}
\begin{enumerate}
\item[\rm($\eta$DM1)]  Whenever $w \in \cWplus$ is not 
mixed-alternating, it follows that $\mu (Z^w) = 0$.
\item[\rm($\eta$DM2)]  Whenever $w = w_1 \cdots w_p \in \cWplus$ is 
mixed-alternating and factored  as in  \emph{(\ref{eqn:27b})},
it follows that $\mu (Z^w) = \mu (Z^{w_1}) \cdots \mu (Z^{w_p})$.
\end{enumerate}
\end{enumerate}
\end{theorem}

The proof of the implication (b) $\Rightarrow$ (a) 
in the above theorem requires two auxilliary results.

\begin{lemma}   \label{lemma:29}
Suppose that a distribution $\mu \in \Dalg (1,*)$ satisfies
the condition \emph{($\eta$DM1)}. Then $\mathrm{Cf}_w ( \eta_{\mu} ) = 0$ for every 
word $w \in \cWplus$ that is not mixed-alternating.
\end{lemma}

\begin{proof} Suppose that $w$ is not mixed alternating and  $|w| = n$. We prove that
$\mathrm{Cf}_w ( \eta_{\mu} ) = 0$  by showing that each term in the right-hand side of
{(\ref{eqn:23d})} vanishes. Indeed, let $\pi = \{ J_1, \ldots , J_m \} \in \Int (n)$
be  a partition 
where the 
intervals $J_1, \ldots , J_m$ are listed in increasing 
order.   Observe that $w = (w | J_1) \cdots (w | J_m)$
(concatenation product).
Since $w$ is not mixed-alternating, there must exist 
an index $1 \leq k \leq m$ such that $w | J_k$ is not 
mixed-alternating.  For this $k$, condition 
($\eta$DM1) yields $\cf_{w | J_k} (M_{\mu}) = 0$. Therefore 
 the term indexed by $\pi$ in {(\ref{eqn:23d})} vanishes as well 
 because  $\cf_{w | J_k} (M_{\mu})$ is one of its factors.  The lemma follows.
\end{proof}

\begin{lemma}   \label{lemma:210}
Let $\mu, \nu \in \Dalg (1,*)$ be such that
\begin{enumerate}
\item[\rm(1)] Both $\mu$ and $\nu$ satisfy conditions 
\emph{($\eta$DM1)} and \emph{($\eta$DM2)}, and
\item[\rm(2)]
\emph{$\cf_w ( \eta_{\mu} ) = \cf_w ( \eta_{\nu} )$} for every
alternating word $w \in \cWplus$.
\end{enumerate}
Then $\mu = \nu$.
\end{lemma}

\begin{proof}  If $w \in \cWplus$  is not
mixed-alternating then $\mu (Z^w) = \nu (Z^w) = 0$ because $\mu$ and $\nu$ satisfy 
 ($\eta$DM1).  Thus it suffices to verify that 
$\mu (Z^w) = \nu (Z^w)$ for mixed-alternating words $w$. In fact, it 
suffices to prove this equality
when $w$ is alternating. Indeed, suppose for the moment that the equality has been proved for alternating words and
 let $w \in \cWplus$ be a mixed-alternating
word.  Consider the canonical factorization $w = w_1 \cdots w_p$ 
indicated in (\ref{eqn:27b}).
We have
\begin{align*}
\mu (Z^w) 
& = \mu (Z^{w_1}) \cdots \mu (Z^{w_p})
     \mbox{  (by ($\eta$DM2) for $\mu$)}                 \\
& = \nu (Z^{w_1}) \cdots \nu (Z^{w_p})
     \mbox{  (by assumption on alternating moments)}     \\
& = \nu ( Z^w )
     \mbox{  (by ($\eta$DM2) for $\nu$).} 
\end{align*}

We conclude the proof by showing that 
$\mu (Z^w  ) = \nu ( Z^w  )$ for every alternating word $w$. By symmetry, it suffices to verify  that $\mu (  (Z^{*}Z)^m ) = \nu (  (Z^{*}Z)^m )$ for every $m\in\mathbb N$
 Fix $m$ and write $\mu (  (Z^{*}Z)^m  )$ and 
$\nu (  (Z^{*}Z)^m  )$ as sums indexed by $\Int (2m)$, 
in the way indicated in  (\ref{eqn:23c}).  We show 
 that for every $\pi \in \Int (2m)$, the terms indexed 
by $\pi$ in the two sums (for $\mu$ and for $\nu$) are 
equal to each other.  If $\pi$ has a block $B$ of odd 
cardinality, then the terms we are looking at are both 
equal to $0$ because they include the factors 
$\cf_{w | B} ( \eta_{\mu} )$ and respectively
$\cf_{w | B} ( \eta_{\nu} )$, and these factors are zero
by Lemma \ref{lemma:29}.
If all the blocks of $\pi$ are even,  we write 
$\pi = \{ J_1, \ldots , J_k \}$ where the intervals 
$J_1, \ldots , J_k$ are listed in increasing order, and where 
$| J_1 | = 2 d_1, \ldots , | J_k | = 2 d_k$ for some 
$d_1, \ldots , d_k \in \bN$.  The terms indexed by $\pi$ in 
the two sums we consider are then
\begin{equation}  \label{eqn:210a}
\prod_{i=1}^k \cf_{ (*,1)^{d_i} } ( \eta_{\mu} )
\mbox{ and respectively }
\prod_{i=1}^k \cf_{ (*,1)^{d_i} } ( \eta_{\nu} )
\end{equation}
where we used the fact that $(*,1)^m | J_i = (*,1)^{d_i}$, 
for $1 \leq i \leq k$.  The two products in 
(\ref{eqn:210a}) are indeed equal by assumption (2) in the statement.
\end{proof}

$ $

\begin{proof}  [Proof of Theorem \emph{\ref{thm:28}}]
Suppose first that $\mu$ is $\eta$-diagonal. 
We verify that it satisfies ($\eta$DM1) and ($\eta$DM2). 
The hypothesis on $\mu$ says in particular 
that $\cf_{w} ( \eta_{\mu} ) = 0$ for every $w \in \cWplus$
which is not mixed-alternating.  To prove ($\eta$DM1),  we 
must show that $\cf_{w} ( M_{\mu} ) = 0$ for every such word.  
This argument is carried precisely as in the proof of 
Lemma \ref{lemma:29}, with the roles of $M_{\mu}$ and 
$\eta_{\mu}$ being reversed and with   (\ref{eqn:23c})
in place of (\ref{eqn:23d}). The reader will have no difficulty 
verifying the details.

In order to show that $\mu$ also satisfies ($\eta$DM2),
fix a mixed-alternating word $w \in \cWplus$, with canonical 
factorization $w = w_1 \cdots w_p$ as in (\ref{eqn:27b}).
Set $n = |w| = |w_1| + \cdots + |w_p|$, and let $\rho_0$ be 
the partition of $\{ 1, \ldots , n \}$ into intervals 
$J_1, \ldots , J_p$ (written in increasing order) with lengths
$|J_1| = |w_1|, \ldots , |J_p| = |w_p|$.  Given a partition $\pi$ of 
$\{ 1, \ldots , n \}$, we  write $\pi \leq \rho_0$
if every block $B$ of $\pi$ is 
contained in one of the blocks $J_1, \ldots , J_p$ of $\rho_0$.
(This relation is usually called the \emph{reverse refinement 
order} on partitions.)

Next, we use  (\ref{eqn:23c}) to express the coefficient
 $\cf_w ( M_{\mu} )=\mu (Z^w)$ as a sum indexed by $\Int (n)$. 
The special structure of the coefficients of $\eta_{\mu}$ implies
that a partition $\pi \in \Int (n)$ has a zero
contribution to that sum unless $\pi \leq \rho_0$.  It is 
immediate that the partitions 
$\pi \in \Int (n)$ satisfying $\pi \leq \rho_0$ are in natural 
bijective correspondence to tuples of partitions 
$( \pi_1, \ldots , \pi_p )$ where 
$\pi_1 \in \Int (J_1), \ldots , \pi_p \in \Int (J_p)$. This 
correspondence is such that for 
$\pi \leftrightarrow ( \pi_1, \ldots , \pi_p)$ we have
\[
\cf_{w; \pi} ( \eta_{\mu} )  
= \cf_{w_1 ; \pi_1} ( \eta_{\mu} ) \cdots 
\cf_{w_p ; \pi_p} ( \eta_{\mu} ).
\]
These observations lead to the formula
\[
\mu (Z^w) =  \sum_{ \begin{array}{c}
{\scriptstyle \pi_1 \in \Int (J_1), \ldots, } \\
{\scriptstyle  \pi_p \in \Int (J_p)} 
\end{array} }  
\cf_{w_1 ; \pi_1} ( \eta_{\mu} ) \cdots 
\cf_{w_p ; \pi_p} ( \eta_{\mu} )
= \prod_{i=1}^p \Bigl( \, \sum_{\pi_i \in \Int (J_i)}
\cf_{w_i ; \pi_i} ( \eta_{\mu} ) \,  \Bigr) .
\]
In the latter product, one more application of 
 (\ref{eqn:23c}) identifies 
\[
\sum_{\pi_i \in \Int (J_i)} 
\cf_{w_i ; \pi_i} ( \eta_{\mu} ) 
= \mu (Z^{w_i}),   1 \leq i \leq p,
\]
thus implying the desired conclusion that 
$\mu ( Z^w ) = \prod_{i=1}^p \mu (Z^{w_i})$.

Conversely, assume now that $\mu$ 
satisfies conditions ($\eta$DM1) and ($\eta$DM2).  We show that it is $\eta$-diagonal by an indirect argument: we construct
an $\eta$-diagonal distribution $\nu \in \Dalg (1,*)$ and then prove that $\mu=\nu$.  The distribution $\nu$ is defined by
specifying its $\eta$-series (see Remark \ref{rem:24}), namely 
$\cf_{ w } ( \eta_{\nu} )=\cf_{ w } ( \eta_{\mu} )$ if $w$ is alternating
and $\cf_{ w } ( \eta_{\nu} )=0$ otherwise. To prove that $\mu=\nu$ we show that
$\mu$ and $\nu$ satisfy the hypothesis of 
Lemma \ref{lemma:210}.  Indeed, both $\mu$ and $\nu$ satisfy 
($\eta$DM1) and ($\eta$DM2): $\mu$ does so by 
hypothesis, while $\nu$ does so because it is $\eta$-diagonal and 
by virtue of the implication (a) $\Rightarrow$ (b) proved 
above.  On the other hand, if $w\in\cWplus$ is an alternating word, the equality 
$\cf_w ( \eta_{\mu} ) = \cf_w ( \eta_{\nu} )$ is true by the definition of $\nu$.
This concludes the proof of the theorem.
\end{proof}

\begin{remark}  \label{rem:211}
Let $( \cA , \varphi )$ be a noncommutative probability space
(that is, $\cA$ is a unital algebra over $\bC$, 
$\varphi : \cA \to \bC$ is a linear functional, and $\varphi (1) = 1$), 
and let $a_1, a_2\in\cA$.  Recall \cite{SW1997} that $a_1, a_2$ 
are said to be 
{\em Boolean independent} provided that, given positive integers
 $n,p_1,\dots,p_n$ and indices $i_1, \ldots , i_n \in \{ 1,2 \}$
such that $i_k\ne i_{k+1}$ for $i=1,\dots,n-1$, the following identity is satisfied:
$$\varphi (a_{i_1}^{p_1} \cdots a_{i_n}^{p_n} )
= \varphi (a_{i_1}^{p_1} ) \cdots \varphi ( a_{i_n}^{p_n} ).$$

Now consider the noncommutative probability space 
$( \bC \langle Z, Z^{*} \rangle , \mu )$, where 
$\mu \in \Dalg (1,*)$ is $\eta$-diagonal.
Condition ($\eta$DM2) of Theorem \ref{thm:28}
can be restated as saying that $ZZ^{*}$ and $Z^{*}Z$ are Boolean 
independent in $( \bC \langle Z, Z^{*} \rangle , \mu )$.
\end{remark}

\begin{remnotation}   \label{rem:211b}
Another relevant fact concerning an $\eta$-diagonal $*$-dis\-tri\-bu\-tion 
$\mu$ concerns the individual  $\eta$-series of $ZZ^{*}$ 
and $Z^{*}Z$ in the noncommutative probability space 
$( \bC \langle Z, Z^{*} \rangle , \mu )$.
If $a$ is an element in a noncommutative probability space 
$( \cA , \varphi )$, then its moment series and $\eta$-series 
$M_a, \eta_a \in \bC [[z]]$ are defined as in Definition \ref{def:23} 
but using moments in place of $*$-moments: first we set
$M_a (z) = \sum_{n=1}^{\infty} \varphi (a^n) z^n$,
and then define
\[
\eta_a (z) = M_a (z) / ( 1+ M_a (z) ) \in \bC [[z]] .
\]
The coefficients of $M_a$ and $\eta_a$ are related to each other
via summations over interval partitions which are analogous to those 
shown in (\ref{eqn:23c}), (\ref{eqn:23d}) (and are derived the same way, by 
starting from the algebraic relations satisfied by the series 
themselves).   We explicitly record here the analogue of 
(\ref{eqn:23d}):
\begin{equation}  \label{eqn:211a}
\mathrm{Cf}_{n}(\eta_a)
= \sum_{\rho \in \Int (n)} (-1)^{1 + | \rho |}
\prod_{B \in \rho} \varphi (a^{ |B| }) , \quad n\in \bN ,
\end{equation}
where (by analogy with Notation \ref{def:22}(2)) we use the
notation $\cf_n : \bC [[z]] \to \bC$ for the linear map 
that extracts the $n$th coefficient of a series in
$\bC [[z]]$.

When applied to the  elements  $ZZ^{*}$ and $Z^{*}Z$ from 
the framework of Theorem \ref{thm:28}, these observations yield 
the following result.
\end{remnotation}

\begin{proposition}   \label{prop:212}
Suppose that $\mu \in \Dalg (1,*)$ is an $\eta$-diagonal distribution,
and let $( \alpha_n )_{n=1}^{\infty}$,
$( \beta_n )_{n=1}^{\infty}$ be its determining sequences
\emph{(}as introduced in Definition \emph{\ref{def:26}(2))}.  Then in the 
noncommutative probability space 
$( \bC \langle Z, Z^{*} \rangle , \mu )$, the elements 
$ZZ^{*}$ and $Z^{*}Z$ have $\eta$-series given by
\[
\eta_{ { }_{ZZ^*} } (z) = \sum_{n=1}^{\infty} \alpha_n z^n
\mbox{ and }
\eta_{ { }_{Z^*Z} } (z) = \sum_{n=1}^{\infty} \beta_n z^n.
\]
\end{proposition}

\begin{proof}
By symmetry, it suffices to prove the first formula.  
Equation (\ref{eqn:211a}) yields
\begin{equation}  \label{eqn:212a}
\cf_{n}(\eta_{ { }_{ZZ^{*}} } )
= \sum_{\rho \in \Int (n)} (-1)^{1 + | \rho |}
\prod_{B \in \rho} \mu (  (ZZ^{*})^{|B|} ),\quad n \in \bN.
\end{equation}
For the remainder of the proof, we fix $n \in \bN$ and  verify that the right-hand side of (\ref{eqn:212a})
is equal to $\alpha_n$.

For every partition 
$\rho = \{ J_1, \ldots , J_k \} \in \Int (n)$, 
with intervals $J_1, \ldots , J_k$ written in increasing
order, we define a \emph{doubled} partition
$\widehat{ \rho } = \{ \widehat{J_1}, \ldots, 
\widehat{J_k}  \} \in \Int (2n)$.
This is the interval partition uniquely determined by the 
requirement that $\widehat{J_1}, \ldots, \widehat{J_k}$
come in increasing order and satisfy
$| \widehat{J_i} | = 2 |J_i|$ for $1 \leq i \leq k$.
With this notation, it is easily seen that the 
right-hand side of  (\ref{eqn:212a}) can be written as
\[
\sum_{\rho \in \Int (n)} (-1)^{1 + | \widehat{\rho} |}
\cf_{ (1,*)^n ; \widehat{\rho} } ( M_{\mu} ).
\]
This, however, is the same as 
\[
\sum_{\pi \in \Int (2n)} (-1)^{1 + | \pi |}
\cf_{ (1,*)^n ; \pi } ( M_{\mu} ).
\]
Indeed, due to the special structure of the $*$-moments of 
$\mu$ described in Theorem \ref{thm:28}, all the  
terms in the latter sum, corresponding to partitions 
$\pi \in \Int (2n)$ which are not of the form 
$\widehat{\rho}$, are equal to $0$. We conclude that 
\[
\cf_{n}(\eta_{ { }_{ZZ^{*}} } )
= \sum_{\pi \in \Int (2n)} (-1)^{1 + | \pi |}
\cf_{ (1,*)^n ; \pi } ( M_{\mu} )
= \cf_{ (1,*)^n } ( \eta_{\mu} )
= \alpha_n,
\]
where  (\ref{eqn:23d}) is used in the second equality.
\end{proof}

\section{$R$-transforms and $R$-diagonal $*$-distributions}

The  discussion in Section 2 is better
put into perspective when one compares it to the parallel
(more elaborate) free probability framework.
In the free probability framework, instead of $\eta$-series one works with 
{\em $R$-transforms}, and one has the concept of what it means 
for a $*$-distribution $\mu \in \Dalg (1,*)$ to be {\em $R$-diagonal}.
The class of $R$-diagonal $*$-distributions is in fact 
rather well-studied in the free probability literature.
Here we review some of their basic properties,
mostly following   \cite[Lecture 15]{NS2006}, and
with emphasis on the facts we need in the present work.

\begin{remark}    \label{rem:31}
On a combinatorial level, 
switching to the world of free probability comes to using
non-crossing partitions instead of interval partitions.
We recall that a \emph{crossing} of a partition $\pi$ of $\{1,\dots n\}$ consists 
of integers $1\le a<b<c<d\le n$ such that the set $\{a,c\}$ is 
contained in a block of $\pi$ and $\{b,d\}$ is contained in
a different block of $\pi$.  A partition is \emph{non-crossing} if it has no crossings. 
We denote by $NC(n)$ the collection of all non-crossing 
partitions of $\{1,\dots n\}$.  

Similarly to the lattice $\text{Int}(n)$, the
set $NC(n)$ is partially ordered by reverse refinement. 
The minimal and maximal elements with respect to this partial
order are denoted by $0_n$ (the partition of $\{ 1, \ldots , n \}$
into $n$ singleton blocks) and respectively $1_n$ (the partition of 
$\{ 1, \ldots , n \}$ into one block).

We record a notation and an elementary observation needed
in the final part of this section. For every 
$n \in \bN$, we denote 
by $NCE(2n)$ the collection of all the partitions $\pi \in NC(2n)$
with the property that every block of $\pi$ has even cardinality.
Observe that if $\pi \in NCE(2n)$ and if 
$V = \{ k_1 < k_2 < \cdots < k_{2m} \}$ is a block of $\pi$,
then the numbers $k_1, k_2, \ldots , k_{2m}$ have alternating 
parities.  Indeed, for every $i=1,\dots, 2m-1$, the set 
$\{ k_i +1 , k_i + 2, \ldots , k_{i+1} - 1 \}$ 
is a union of blocks of $\pi$, and hence has even cardinality, 
which implies that $k_{i+1}$ is of opposite parity from $k_i$.

For a discussion of other elementary facts
concerning $NC(n)$, we refer to \cite[Lecture 9]{NS2006}.
\end{remark}

\begin{remark}    \label{def:32}
({\em Review of $R$-transforms.})

(1) The {\em $R$-transform} of a $*$-distribution 
$\mu \in \Dalg (1,*)$ is the series $R_{\mu} \in \ncsercon$, whose 
coefficients are uniquely determined by the requirement that they 
relate to the $*$-moments of $\mu$ by the formula
\begin{equation}  \label{eqn:32a}
\cf_{w} (M_{\mu}) = \sum_{\pi \in NC(n)} 
\cf_{w ; \pi} (R_{\mu}),\quad  
w \in \cWplus \mbox{ and } n=|w|.
\end{equation}
Equation (\ref{eqn:32a}) is the free probabilistic counterpart 
of (\ref{eqn:23c}).  It is 
often referred 
to as the \emph{moment-cumulant formula} for free cumulants (see 
\cite[Lecture 11]{NS2006} for an explanation of this terminology). 

One can also define the series $R_{\mu}$ by an equation 
involving the series $M_{\mu}$ and $R_{\mu}$ themselves (rather than 
 their coefficients, as in (\ref{eqn:32a})).
More precisely, $R_{\mu}$ is the unique series in $\ncsercon$
that satisfies the functional equation
\begin{equation}  \label{eqn:32b}
R_\mu(z(1+M_\mu(z,z^*)),z^*(1+M_\mu(z,z^*))=M_\mu(z,z^*)
\end{equation}
(see \cite[Corollary 16.16]{NS2006}).
This is the free probabilistic analogue of 
 (\ref{eqn:23a}), but
now we only have an implicit functional equation rather
than an explicit formula describing the series $R_{\mu}$.

An easy inductive argument in the moment-cumulant 
formula (\ref{eqn:32a}) shows
that one can recover $M_{\mu}$ from $R_{\mu}$ and that 
(as in  Remark \ref{rem:24}) we have a bijection
\[
\Dalg (1,*) \ni \mu \mapsto R_{\mu} \in \ncsercon .
\]
In other words, one can define a distribution 
$\mu \in \Dalg (1,*)$ by specifying its $R$-transform.

\vspace{6pt}

(2) Consider the framework of Remark \ref{rem:211}, where 
we discussed the moment series and $\eta$-series 
$M_a, \eta_a \in \bC [[z]]$ associated to an element $a$ 
in a noncommutative probability 
space $( \cA , \varphi )$.  In that framework one also has
an {\em $R$-transform} associated with 
the element $a \in \cA$. This is the series $R_a \in \bC [[z]]$ 
that relates to $M_a$ by
\begin{equation}   \label{eqn:32d}
\cf_{n} (M_{a}) = \sum_{\pi \in NC(n)} 
\cf_{n ; \pi} (R_{a}),\quad  n \in \bN,
\end{equation}
and
\begin{equation}   \label{eqn:32e}
R_a ( z(1+ M_a (z)) ) = M_a (z).
\end{equation}
These formulas are analogous to
(\ref{eqn:32a}) or (\ref{eqn:32b}), respectively.
For a detailed discussion of the algebraic aspects of $R$-transforms
(covering both the series $R_{\mu}$ in part (1) of this remark
and the series $R_a$ in part (2)), see  \cite[Lecture 16]{NS2006}.
\end{remark}

\begin{definition} \label{def:33}
(1) A $*$-distribution $\mu\in\Dalg(1,*)$ is said to be 
{\em $R$-diagonal} when $\cf_w ( R_{\mu } ) = 0$ for every word 
$w \in \cWplus$ that is not alternating.

(2) If $\mu$ is $R$-diagonal, its $R$-transform is
thus of the form
\begin{align}    \label{eqn:33a}
R_{\mu} (z,z^{*}) = \sum_{n=1}^{\infty} \alpha_n (zz^{*})^n
+ \sum_{n=1}^{\infty} \beta_n (z^{*}z)^n,
\end{align}
with $\alpha_n, \beta_n \in \bC$ for $n \in \bN$.  The
sequences $( \alpha_n )_{n=1}^{\infty}$ and 
$( \beta_n )_{n=1}^{\infty}$ are called the 
{\em determining sequences} of $\mu$.
\end{definition}

\begin{remark}   \label{rem:34}
The concept of an $\eta$-diagonal $*$-distribution from 
 Section 2 obviously parallels the one of an $R$-diagonal 
$*$-distribution, with the $\eta$-series in place of the 
$R$-transform.  The basic properties of 
$\eta$-diagonal distributions proved in Section 2 are the 
counterparts of known facts concerning 
$R$-diagonal distributions, as noted below.

(1) Remark \ref{rem:211} is the 
Boolean counterpart of \cite[Corollary 15.11, p. 244]{NS2006}: 
if $\mu$ is $R$-diagonal, then $Z^{*}Z$ and $ZZ^{*}$ are freely
independent in the noncommutative probability space 
$(\mathbb{C}\langle Z,Z^*\rangle,\mu)$.  

(2) Theorem \ref{thm:28} is the counterpart 
of \cite[Theorem 1.2.1]{NSS2001} which describes $R$-diagonal 
distributions in terms of their $*$-moments. 

(3) Proposition \ref{prop:212}
is analogous to \cite[Proposition 15.6, p. 241]{NS2006} which 
gives a precise formula, first found in \cite{KS2000},
for the coefficients of the one-variable $R$-transforms
$R_{ {ZZ^*} }$ and $R_{ {Z^{*}Z} }$ in terms of the 
determining sequences
$(\alpha_n)_{n=1}^{\infty}$ and $(\beta_n)_{n=1}^{\infty}$ from 
(\ref{eqn:33a}).  This formula is more elaborate than 
the relation found 
in Proposition \ref{prop:212} for $\eta$-diagonal elements. It states that
\begin{equation}    \label{eqn:34a}
{\rm Cf}_n ( R_{{ZZ^*} } )= \sum _{
{ \pi = \{ V_1, \ldots , V_k \} \in NC(n)} }
\alpha_{|V_1|} \beta_{|V_2|} \cdots \beta_{|V_k|},
\end{equation}
where the blocks of $\pi$ are arranged so $1\in V_1$. The coefficients of 
$R_{{Z^{*}Z} }$ are obtained by interchanging the roles of 
$\alpha$ and $\beta$ in these formulas. 
For example, the first three coefficients of
$R_{{ZZ^{*}} }$ are $\alpha_1, \alpha_2 + \alpha_1 \beta_1$
and $\alpha_3 + 2 \alpha_2 \beta_1 + \alpha_1 \beta_2 
+ \alpha_1 \beta_1^2$. 

Equations (\ref{eqn:34a}) lead to the following observation: an easy 
induction on $n$ 
(where one singles out the terms indexed by the partition 
$1_n \in NC(n)$ on the right-hand sides of 
these equations) shows that the determining sequences
$( \alpha_n )_{n=1}^{\infty}$ and $(\beta_n )_{n=1}^{\infty}$
can be retrieved from the coefficients of 
$R_{ {ZZ^{*}} }$ and $R_{ {Z^{*}Z} }$. 
Hence the $R$-diagonal
$*$-distribution $\mu \in \Dalg (1,*)$ is completely determined 
by these $R$-transforms.

In Section 6  we require a reformulation of 
(\ref{eqn:34a}) in terms of operations with series rather than individual 
coefficients.  This reformulation is 
given in the next proposition.  The formulas 
(\ref{eqn:36a}) bear a striking 
resemblance to the functional equation (\ref{eqn:32e}) of the $R$-transform.
  In fact, (\ref{eqn:36a}) collapse to
$R_{{ZZ^{*}} } (z) = R_{{Z^{*}Z} } (z) = M_a (z)$ in the special case
$\alpha_n = \beta_n$, $n \in \bN$, in which case we can
take $a = b$.
\end{remark}

\begin{proposition}      \label{prop:36}
Let $\mu \in \Dalg (1,*)$ be an $R$-diagonal $*$-distribution 
with determining sequences $( \alpha_n )_{n=1}^{\infty}$ and
$( \beta_n )_{n=1}^{\infty}$.  Suppose we are given some elements 
$a$ and $b$ in a noncommutative probability space $( \cA , \varphi )$ 
such that
\[
R_a (z) = \sum_{n=1}^{\infty} \alpha_n z^n \text{ and }   
R_b (z) = \sum_{n=1}^{\infty} \beta_n z^n.   
\]
Then, in the noncommutative probability space 
$( \bC \langle Z, Z^{*} \rangle , \mu )$, we have
\begin{equation}  \label{eqn:36a}
R_{ {ZZ^{*}} } (z) 
  = R_a ( z( 1 + M_b (z)) ) \text{ and }                                 
R_{ {Z^{*}Z} } (z) 
  = R_b( z( 1 + M_a (z))) .
\end{equation}
\end{proposition} 

\begin{proof} 
The argument  is analogous to the proof of (\ref{eqn:32b}) (see, for instance, the 
proof of \cite[Theorem 16.15]{NS2006}).
For the reader's convenience, we describe the basic idea.
  
By symmetry, it suffices to prove
the first equality in (\ref{eqn:36a}). We show that the coefficients of $z^n$
in the series
$R_{{ZZ^{*}} } (z)$ and $R_a ( z( 1 + M_b (z)) )$ are equal to each other for every $n\in\mathbb N$. The formal series expansion
$$R_a ( z( 1 + M_b (z)) )
= \sum_{m=1}^{\infty} \alpha_m \bigl( z(1 + M_b (z)) \bigr)^m$$
yields
$$
\cf_n  \bigl( R_a ( z( 1 + M_b (z)) ) \bigr) 
= \sum_{m=1}^n \alpha_m   \cf_{n-m} \bigl( (1+ M_b)^m \bigr),
\quad n\in\mathbb N.
$$
Recall that the coefficients of
$1 + M_b$ are moments of $b$, and expand $(1+ M_b)^m$, to obtain
\begin{equation}   \label{eqn:36b}
\cf_n [  R_a (  z(1+M_b) (z) )  ]
= \sum_{m=1}^n   \sum_{\begin{array}{c}
{\scriptstyle k_1, \ldots , k_m \geq 0}  \\
{\scriptstyle \text{with }  k_1 + \cdots + k_m = n-m}
\end{array}  } \alpha_m 
 \varphi ( b^{k_1} ) \cdots \varphi ( b^{k_m} )
 .
\end{equation}
On the other hand,  (\ref{eqn:34a}) yields
\[
\cf_n ( R_{ {ZZ^{*}} } ) = \sum_{m=1}^n
 \sum_{\begin{array}{c}
\scriptstyle S \subseteq \{ 1, \ldots , n \} 
\text{ with}
\\
\scriptstyle |S| = m \text{ and }1\in S
\end{array}  }  
\sum_{ \begin{array}{c}
{\scriptstyle \pi = \{ V_1, \ldots , V_p \} \in NC(n) } \\
{\scriptstyle \text{with } V_1 = S}
\end{array} } \alpha_m  \beta_{|V_2|} \cdots \beta_{|V_p|},\quad n\in\mathbb N.
\]
For a fixed set 
$S = \{ s_1, \ldots , s_m \} \subseteq \{ 1, \ldots , n \}$
with $1 = s_1 < s_2 < \cdots < s_m \leq n$, the collection 
of non-crossing partitions 
$\{ \pi \in NC(n) : S \mbox{ is a block of $\pi$} \}$
is naturally identified with the Cartesian product 
\[
NC( s_2 -s_1 -1) \times \cdots \times
NC( s_m -s_{m-1} -1) \times NC ( n - s_m ),
\]
in a way that converts the sum 
\[
\sum_{ \begin{array}{c}
{\scriptstyle \pi = \{ V_1, \ldots , V_p \} \in NC(n) } \\
{\scriptstyle \text{with } V_1 = S}
\end{array} } \ \beta_{|V_2|} \cdots \beta_{|V_p|}
\]
into the product
\begin{equation}   \label{eqn:36c}
\prod_{\ell=1}^{m}\left(  \sum_{\pi_\ell \in NC(s_{\ell+1} - s_{\ell} -1)} 
         \prod_{V \in \pi_\ell} \beta_{|V|}  \right),
\end{equation}
where we set $s_{m+1}=n+1$.
An application of the moment-cumulant formula (\ref{eqn:32d}) 
 shows that $\prod_{V \in \pi_\ell} \beta_{|V|}=\varphi(b^{s_{\ell+1}-s_\ell-1})$.
(See the  proof of \cite[Theorem 16.15]{NS2006} for more details.) We conclude that 
\begin{equation}\label{eqn:36d}\cf_n ( R_{{ZZ^{*}} } )=
\sum_{m=1}^n \   \sum_{ 
\begin{array}{c}
{\scriptstyle S \subseteq \{ 1, \ldots , n \} \text{ with}}  \\
{\scriptstyle |S| = m \text{ and }1\in S}
\end{array}  } \alpha_m \prod_{\ell=1}^m 
\varphi ( b^{s_{\ell+1} - s_\ell - 1} ),\quad n\in\mathbb N.
\end{equation}
Finally,  observe that for every fixed 
$m \in \{ 1, \ldots , n \}$ there is a natural bijection
between tuples $( k_1, \ldots , k_m ) \in ( \bN \cup \{ 0 \} )^m$
with $k_1 + \cdots + k_m = n-m$ (on the one hand) and subsets 
$1 \in S \subseteq  \{ 1, \ldots , n \}$ with $|S| = m$
(on the other), given by the formula
\[
(k_1, \ldots , k_m) \mapsto
S = \{ 1, k_1 + 2, k_1 + k_2 + 3, \ldots ,
k_1 + \cdots + k_{m-1} + m  \}.
\]
The inner sums on the right-hand 
sides of  (\ref{eqn:36b}) and (\ref{eqn:36d}) are 
identified term by term via this bijection, and this concludes the 
proof.
\end{proof}

In the remainder of this section we discuss the $R$-diagonal 
$*$-distributions that satisfy the KMS condition.  This is a 
special case of the class of 
$*$-distribution studied in \cite{Sh1997} (see, for instance
\cite[Remark 2.10]{Sh1997}).  The best known example of a KMS 
$R$-diagonal distribution is the one where, in the framework 
of the next definition, one sets 
$\alpha_1 = \lambda$, $\beta_1 = 1$ and
$\alpha_n = \beta_n = 0$ for $n \geq 2$; this is called the
$\lambda$-\emph{circular distribution}, and is studied in detail 
in  \cite[Section 4]{Sh1997}.

\begin{definition}    \label{def:36}
Let $\mu \in \Dalg (1,*)$ be an $R$-diagonal $*$-distribution 
with determining sequences $( \alpha_n )_{n=1}^{\infty}$ and
$( \beta_n )_{n=1}^{\infty}$, and let $t$ be a positive real number.
We say that $\mu$ satisfies the {\em KMS condition} with 
parameter $t$ if 
\begin{equation}   
\alpha_n = t \beta_n , \quad n \in \bN .
\end{equation}
\end{definition}

The following result shows that the KMS condition is a 
generalization of the trace property, where the latter property
occurs for the value $t = 1$ of the parameter.

\begin{proposition}   \label{prop:37}
Let $\mu \in \Dalg (1,*)$ be an $R$-diagonal $*$-distribution,
let $t$ be a positive real number, and suppose that $\mu$
satisfies the KMS condition with parameter $t$.  Denote by 
$$U_t : \bC \langle Z, Z^{*} \rangle 
   \to \bC \langle Z, Z^{*} \rangle$$
the unique unital algebra homomorphism such that 
\[
U_t (Z) = tZ \mbox{ and } U_t (Z^{*}) = \frac{1}{t} Z^{*}.
\]
Then 
\begin{equation}   \label{eqn:37a}
\mu ( P  Q ) = \mu ( Q  U_t (P)), \quad 
 P,Q \in \bC \langle Z, Z^{*} \rangle .
\end{equation}
\end{proposition}

\begin{proof}  Both sides of  (\ref{eqn:37a}) are 
bilinear in $P$ and $Q$, so it suffices to check the equation
when both $P$ and $Q$ are monomials. Using the notation $Z^\varnothing=1$, we must show that
\begin{equation}   \label{eqn:37b}
\mu ( Z^v  Z^w ) = \mu ( Z^w \, U_t ( Z^v ) ),\quad
 v,w \in \cWplus \cup \{ \varnothing \}.
\end{equation}  
Equivalently, we must show that the set
\[
\cS =\{ v \in \cWplus \cup \{ \varnothing \} : 
\mu ( Z^v  Z^w ) = \mu ( Z^w  U_t ( Z^v ) ),
 w \in \cWplus \cup \{ \phi \}  \}
\]
is equal to $\cWplus \cup \{ \varnothing \}$.  The set $\mathcal S$ is clearly closed under concatenation and contains $\varnothing$.  Therefore, it suffices to show that $\{1,*\}\subset \mathcal S$. In other words it suffices to prove that
\begin{equation}   \label{eqn:37c}
\mu ( Z  Z^w ) = t  \mu ( Z^w  Z )
\mbox{ and }
\mu ( Z^{*}  Z^w ) = \frac{1}{t}  \mu ( Z^w  Z^{*} ),
\quad w \in \cWplus \cup \{ \varnothing \}.
\end{equation}
We  only prove the first equality in (\ref{eqn:37c});
the verification of the second one is analogous.  The case 
$w = \varnothing$ follows from the fact (incorporated in the definition 
of an $R$-diagonal distribution) that $\mu (Z) = 0$.  For the 
remainder part of the proof we fix a word 
$w  = ( \ell_1, \ldots , \ell_n ) \in \cWplus$, for 
which we prove that 
$\mu ( Z  Z^w ) = t \, \mu ( Z^w  Z )$.
Observe that
\[
\mu ( Z  Z_w ) = \cf_{w_1} (M_{\mu}), \quad 
\mu ( Z_w  Z ) = \cf_{w_2} (M_{\mu}),  
\]
where $w_1 := (1, \ell_1, \ldots , \ell_n )$ and 
$w_2 := ( \ell_1, \ldots , \ell_n, 1 )$.  It is convenient to 
view $w_1$ and $w_2$ as functions from $\{ 1, \ldots , n+1 \}$ 
to $\{ 1, * \}$ and to record the fact that 
\begin{equation}   \label{eqn:37d}
w_2 = w_1 \circ \gamma_{n+1},
\end{equation}
where $\gamma_{n+1}$ is the cyclic permutation 
$1 \mapsto 2 \mapsto \cdots \mapsto n+1 \mapsto 1$ of 
$\{ 1, \ldots , n+1 \}$.

For every partition 
$\pi = \{ V_1, \ldots , V_k \} \in NC(n+1)$ we denote
by $\gamma_{n+1}^{-1}( \pi)$ the partition (still in 
$NC(n+1)$) whose blocks are
$\gamma_{n+1}^{-1} (V_1), \ldots , \gamma_{n+1}^{-1} (V_k)$.
The desired conclusion
$\cf_{w_1} ( M_{\mu} ) = t  \cf_{w_2} ( M_{\mu} )$ is obtained  from
\begin{equation}   \label{eqn:37e}
\cf_{w_1 ; \pi} ( R_{\mu} ) 
= t  \cf_{w_2 ; \gamma_{n+1}^{-1}( \pi)} ( R_{\mu} ),
\quad \pi \in NC(n+1),
\end{equation}  
using the moment-cumulant formula. Indeed, sum both sides of 
(\ref{eqn:37e}) over $\pi \in NC(n+1)$   and invoke  
(\ref{eqn:32a}) applied to the words $w_1$ and $w_2$. The sums thus obtained 
are precisely $\cf_{w_1} (M_{\mu})$ on the left side and
$t \cf_{w_2} (M_{\mu})$ on the right.

Thus, it remains to prove (\ref{eqn:37e}). Fix a partition 
$\pi = \{ V_1, \ldots , V_k \}\in NC (n+1)$ such that $1\in V_1$.
Then $\gamma_{n+1}^{-1} ( \pi) = \{ W_1, \ldots , W_k \}$,
where $W_j = \gamma_{n+1}^{-1} (V_j)$, $1 \leq j \leq k$, 
and  $n+1\in W_1$.  It follows from (\ref{eqn:37d}) that
$w_1 \mid V_j = w_2 \mid W_j \in \cWplus$ if $2\le j\le k$, and thus
\begin{equation}   \label{eqn:37f} 
\cf_{w_1 \mid V_j} (R_{\mu}) = \cf_{w_2 \mid W_j} (R_{\mu}),
\quad 2 \leq j \leq k
\end{equation}
For the remaining block, we show that 
\begin{equation}   \label{eqn:37g} 
\cf_{w_1 \mid V_1} (R_{\mu}) 
= t \cdot \cf_{w_2 \mid W_1} (R_{\mu}).
\end{equation}
Indeed, suppose that $V_1 =  \{ j_1, \ldots , j_m \}$
with $1 = j_1 < j_2 < \cdots < j_m$, and therefore  
$W_1 = \{ j_2 -1 , \ldots , j_m - 1, n+1 \}$.  Both sides of 
(\ref{eqn:37g}) are $0$ if $m$ is odd or if $m$ is even but
$w_1 | V_1$ is not an alternating word. 
If $m$ is even and $w_1|V_1$ is alternating, then we find 
that $w_1 | V_1 = \{ 1,* \}^{m/2}$ and 
$w_2 | V_2 = \{ *,1 \}^{m/2}$, which implies that 
$\cf_{w_1 \mid V_1} (R_{\mu}) = \alpha_{m/2}$ and 
$\cf_{w_2 \mid W_1} (R_{\mu}) = \beta_{m/2}$. In this case, 
(\ref{eqn:37g}) follows from the KMS hypothesis.

Finally, for the partition $\pi$ fixed in the preceding 
paragraph we write:
\begin{align*}
\cf_{w_1 ; \pi }  (R_{\mu})
& = \cf_{w_1 \mid V_1} (R_{\mu})
\prod_{j=2}^k \cf_{w_1 \mid V_j} (R_{\mu})         \\
& = t  \cf_{w_2 \mid W_1} (R_{\mu})
\prod_{j=2}^k \cf_{w_2 \mid W_j} (R_{\mu})     
\mbox{ (by (\ref{eqn:37f}) and (\ref{eqn:37g})) }  \\
& = t  \cf_{w_2 ; \gamma_{n+1}^{-1} ( \pi )} (R_{\mu}),
\end{align*}
thus concluding the proof of (\ref{eqn:37e}).
\end{proof}

\begin{remark}    \label{rem:38}
The converse of Proposition \ref{prop:37} is also true.  More 
precisely, every $R$-diagonal distribution $\mu \in \Dalg (1,*)$ 
that satisfies (\ref{eqn:37a}) for some $t\in (0,+\infty)$ must 
also satisfy the KMS condition for the same value of $t$. To see 
this, let  $( \alpha_n )_{n=1}^{\infty}$ and 
$( \beta_n )_{n=1}^{\infty}$ be the determining sequences of 
$\mu$. Equation (\ref{eqn:37a}) yields, in particular, the identity 
\begin{equation}  \label{eqn:38a}
\mu ( (ZZ^{*})^n ) 
= t \mu ( (Z^{*}Z)^n  ), \quad n\in\mathbb N.
\end{equation}
This identity implies $\alpha_n =t\beta_n$, $n\in\mathbb N$,
by induction on $n$. For the induction step one invokes the 
moment-cumulant formula in order to expand both
sides of (\ref{eqn:38a}) as sums over $NC(2n)$;
then the action of the cyclic permutation $\gamma_{2n}^{-1}$ on 
$NC(2n)$ can be used in the same way as it was done in the proof 
of Proposition \ref{prop:37}.
\end{remark}

\section{The framework of $\cD_c (1,*)$ and $\cD_c (k)$,
BBP bijections}

We now introduce the analytic framework which is 
of interest for the present paper. 

\begin{definition}  \label{def:41}
(1) Let $( \cA , \varphi )$ be a $C^{*}$-probability space
(which means that $\cA$ is a unital $C^{*}$-algebra,  
$\varphi : \cA \to \bC$ is a positive linear functional, and 
 $\varphi (1) = 1$), and let $a\in\cA$.  The {\em $*$-distribution of $a$} is the
functional $\mu \in \Dalg (1,*)$  determined by the 
requirement that 
\[
\mu (Z^w ) =  \varphi  (a^w ),\quad w\in\mathcal{W}^+.
\]

(2) We denote by $\cD_c (1,*)$ the set of all 
elements of $\Dalg (1,*)$ that are equal to the $*$-distribution of some element in a $C^{*}$-probability space.

(3) Free additive (respectively, multiplicative) convolution is a binary operation on $\cD_c (1,*)$  
denoted by $\boxplus$ (respectively, $\boxtimes$). This operation is uniquely determined by the following property: given elements $a, a '$ in a $C^{*}$-probability space 
$( \cA , \varphi )$ such that $\{ a ,a^{*} \}$ is free from 
$\{ a', ( a')^{*} \}$, the $*$-distribution of $a+a'$ (respectively, $aa'$) is the free additive (respectively, multiplicative) convolution of the $*$-distributions of $a$ and $a'$.  See \cite[Lectures 5 and 7]{NS2006} for more details.

(4) An element $\mu\in\Dc (1,*)$ is said to be $\boxplus$-\emph{infinitely 
divisible} if for every $n\in\mathbb N$ there exists 
$\mu_n\in \Dc (1,*)$ such that 
\[
\mu= \underbrace{\mu_n \boxplus \cdots \boxplus \mu_n}_{\mbox{ $n$  times}}.
\]
The set of all $\boxplus$-infinitely divisible distributions in $\Dc (1,*)$ 
is denoted by $\Dcinfdiv (1,*)$.
\end{definition}

The main result of this section is the following theorem.  The 
series $R_\nu$ and $\eta_\mu$ appearing in the statement of the 
theorem are as defined in Sections 2 and 3.

\begin{theorem}   \label{thm:42}
{\em (BBP bijection on $\cD_c (1,*)$.) }
There exists a bijection 
$$\bB_{(1,*)}:\Dc (1,*)\to\Dcinfdiv (1,*),$$
determined by the requirement that
\begin{equation}    \label{eqn:42a}
R_{\bB_{(1,*)}(\mu)}=\eta_\mu, \quad \mu \in \Dc (1,*).
\end{equation}
More precisely, for every $\mu \in \cD_c (1,*)$ there exists a unique $*$-distribution 
$\nu \in \Dcinfdiv (1,*)$ such that 
$R_{\nu} = \eta_{\mu}$, and we define $\bB_{(1,*)} ( \mu ) := \nu$.
\end{theorem}

\begin{defrem}   \label{def:43}
({\em Framework of $\cD_c (k)$.})
We  reduce Theorem \ref{thm:42}  to an analogous 
theorem proved in \cite{BN2008} for the space, denoted 
by $\cD_c (2)$, of joint distributions of pairs of 
selfadjoint elements in a $C^{*}$-probability space.  
The passage from $\cD_c (1,*)$ to $\cD_c (2)$ is natural, and 
essentially amounts to the change of variables
\[
(a, a^{*} ) \mapsto 
\left( \frac{a+a^{*}}{2}, \frac{a-a^{*}}{2i} \right),   
\]
for $a$ in a $C^{*}$-probability space $( \cA , \varphi )$. In order to clarify this idea, we  
review briefly the framework 
of $\cD_c (k)$. Fix $k\in\mathbb N$.

(1) We denote by $\ncpolk$ the algebra of  polynomials 
in the non-commuting indeterminates $X_1,\ldots,X_k$. 

(2)  Let $( \cA , \varphi )$ be a $C^{*}$-probability space and let
$b_1, \ldots , b_k\in\cA$ be selfadjoint.  The {\em joint
distribution} of $b_1, \ldots , b_k$ is the linear functional 
$\lambda : \ncpolk \to \bC$ which is determined by the requirement 
that $\lambda (1) = 1$ and 
\[
\lambda ( X_{i_1} \cdots X_{i_n} ) = \varphi ( b_{i_1} \cdots b_{i_n} ),
  \quad n \in \bN, 
i_1, \ldots , i_n \in \{ 1, \ldots , k \} .
\]

(3) We denote by $\cD_c (k)$ the set of all linear 
functionals $\lambda : \ncpolk \to \bC$ that can arise as joint 
distributions of $k$-tuples of selfadjoint elements in some 
$C^*$-probability space.

(4) Free additive convolution is a binary operation on $\cD_c (k)$  
denoted
\footnote{It is customary to always denote free additive 
convolution by ``$\boxplus$''.  The setting in which the symbol 
$\boxplus$ is used should be clear, in each case, from the context.}
by 
$\boxplus$. This operation is uniquely determined by the following 
property: given selfadjoint elements $a_1,\dots,a_k$ and 
$b_1,\dots,b_k$ in a $C^{*}$-probability space $( \cA , \varphi )$ 
such that $\{a_1,\dots,a_k\}$ is free from $\{b_1,\dots,b_k\}$, the 
joint distribution of $a_1+b_1,\dots,a_k+b_k$ is the free additive 
convolution of the joint distributions of $a_1,\dots,a_k$ and 
$b_1,\dots,b_k$. 
 The  concept of $\boxplus$-infinite 
divisibility in $\cD_c (k)$ is introduced as in Definition \ref{def:41}(4).
The set of $\boxplus$-infinitely 
divisible  distributions in $\cD_c (k)$ is denoted by $\Dcinfdiv (k)$.

(5) We denote by $\ncserk$ the space of those
formal power series with complex coefficients in $k$ 
non-commuting indeterminates $x_1, \ldots , x_k$ whose 
constant term is equal to $0$.  We denote by $\cf_{ (i_1, \ldots , i_n) } (f)$ 
the coefficient of $x_{i_1} \cdots x_{i_n}$ in a series $f \in \ncserk$.
Every joint distribution $\lambda \in \cD_c (k)$ has 
a moment series $M_{\lambda}$, an $R$-transform $R_{\lambda}$ and an 
$\eta$-series $\eta_{\lambda}$.  These are elements of $\ncserk$, 
and their definitions are analogous to Definitions 
\ref{def:23} and \ref{def:32}.  A detailed description of these 
power series and of the relations between their coefficients can be 
found in \cite[pp. 14-17]{BN2008}.
\end{defrem}

The proof of  Theorem \ref{thm:42} will be reduced to the following 
result from \cite{BN2008} (see also \cite{BP1999} for the case $k=1$).

\begin{theorem}   \label{thm:44}
{\em (BBP bijection on $\cD_c (k)$.)}
Let $k$ be a positive integer.
There exists a bijection $\bB_k : \Dc (k) \to\Dcinfdiv (k)$,
determined by the requirement that
\begin{equation}    \label{eqn:44a}
R_{\bB_k (\lambda)}=\eta_{\lambda}, \, \quad \lambda \in \Dc (k).
\end{equation}
More precisely, for every $\lambda \in \cD_c (k)$ there exists a unique $\nu \in \Dcinfdiv (k)$  such that 
$R_{\nu} = \eta_{\lambda}$, and we define $\bB_k ( \lambda ) := \nu$.
\hfill $\square$
\end{theorem}

\begin{remark}   \label{rem:45}
We only need Theorem \ref{thm:44} for $k=1$ and $k=2$. 
When 
$k=1$, the space $\cD_c (1)$ is naturally identified with the space 
$\cP_c$ of compactly supported Borel probability measures on $\bR$.  
Indeed, given $b=b^{*}$ in a $C^{*}$-probability space
$( \cA , \varphi )$, 
Definition \ref{def:43}(2) produces a linear functional 
$\lambda : \bC \langle x_1 \rangle \to \bC$ which becomes, via 
the  Riesz representation theorem, a Borel probability measure 
supported on the spectrum of $b$.  The original BBP bijection from 
\cite{BP1999} was defined on $\cP_c$ (and on the larger set $\cP$ 
of all Borel probability measures on $\mathbb R$).  In Section 6, 
we will simply talk about $\bB ( \sigma )$ for $\sigma \in \cP_c$. 
In other words if $\lambda$ denotes the functional in $\cD_c (1)$ 
corresponding to $\sigma$, then $\bB ( \sigma ) \in \cP_c$  
denotes the probability measure corresponding to $\bB_1 ( \lambda )$.

The following result creates bijections $C$ and $D$ that we use in conjunction with the case $k=2$ of Theorem \ref{thm:44}. (The letters $C$ and $D$ are meant to suggest  
\emph{complexification} and \emph{decomplexification}.)
The proof is immediate, 
and therefore omitted.
\end{remark}

\begin{proposition}   \label{prop:46}
There exists a bijection
$D : \cD_c (1,*) \to \cD_c (2 )$ defined as follows.
Given  $\mu \in \cD_c ( 1,*)$ that is the $*$-distribution of
an element $a$ in a $C^{*}$-probability space 
$( \cA , \varphi )$, $D(\mu)$ is the joint distribution of the pair
$$
\left(\frac{a+a^{*}}{2}, \frac{a-a^{*}}{2i}\right).
$$
The inverse of $D$ is the bijection 
$C : \cD_c (2) \to \cD_c (1,*)$ defined as follows. Given  $\lambda \in \cD_c (2)$ that is the joint distribution of a pair
 $b_1, b_2$ of selfadjoint elements in a 
$C^{*}$-probability space $( \cA , \varphi )$, $C(\lambda)$ is the 
$*$-distribution of $b_1+ib_2$.
\hfill $\square$
\end{proposition}

\begin{defrem}   \label{def:47}
In addition to the transformations $C$ and $D$, the proof of Theorem \ref{thm:42} requires the corresponding change of variables for power series. Denote by
\begin{equation}   \label{eqn:47a}
t_{1,1} = t_{1,*} = \frac{1}{2}  \mbox{ and } 
t_{2,1} = \frac{1}{2i},  t_{2,*} = - \frac{1}{2i},
\end{equation}
the coefficients of the linear transformation 
$b_1 = (a+ a^{*})/{2}$, $b_2 = (a - a^{*})/{2i}$. This 
transformation can now be written more compactly as
\begin{equation}    \label{eqn:47b}
b_i = \sum_{ \ell \in \{ 1,* \} }
 t_{i, \ell} a^{\ell},   \mbox{ for $i = 1,2$.}
\end{equation}
Using the coefficients $t_{i, \ell}$ we define a map $\widetilde{D} : 
\bC_0 \langle \langle z,z^{*} \rangle \rangle \to
\bC_0 \langle \langle x_1, x_2 \rangle \rangle$
as follows: given a series 
$f \in \bC_0 \langle \langle z,z^{*} \rangle \rangle$,
the coefficients of the  series $g = \widetilde{D} (f) \in
\bC_0 \langle \langle x_1, x_2 \rangle \rangle$
are given by
\begin{equation}   \label{eqn:47c}
\cf_{ (i_1, \ldots , i_n) } (g) 
:= \sum_{\ell_1, \ldots , \ell_n \in \{ 1,* \} }  
t_{i_1, \ell_1} \cdots  t_{i_n , \ell_n} 
\cf_{ ( \ell_1, \ldots , \ell_n ) } (f),\quad n \in \bN,i_1, \ldots , i_n \in \{ 1,2 \}.  
\end{equation}
The map $\widetilde{D}$ is clearly linear and bijective.  Its 
inverse $\widetilde{C} :
\bC_0 \langle \langle x_1, x_2 \rangle \rangle
\to \bC_0 \langle \langle z,z^{*} \rangle \rangle$ is
defined by a formula analogous to (\ref{eqn:47c}),
but with $[t_{i, \ell}]$  replaced by the inverse matrix
\[
t'_{1,1} = 1,  t'_{1,2} =  i \mbox{ and }
t'_{*,1} = 1,  t'_{*,2} = -i.
\]
\end{defrem}

\begin{lemma}  \label{lemma:48}
For every $*$-distribution 
$\mu \in \cD_c (1,*)$, we have
\begin{equation}    \label{eqn:48a}
M_{D ( \mu ) } = \widetilde{D} ( M_{\mu} ),
R_{D ( \mu ) } = \widetilde{D} ( R_{\mu} ),
\text{ and }
\eta_{D ( \mu ) } = \widetilde{D} ( \eta_{\mu} ).
\end{equation}
\end{lemma} 

\begin{proof} Suppose that  $\mu \in \cD_c (1,*)$ is the $*$-distribution of an element $a$ in a $C^{*}$-probability 
space $( \cA , \varphi )$, and set $b_i=\sum_{\ell\in\{1,*\}}t_{i,\ell}a^\ell$ for $i=1,2$. By definition, the joint distribution of $b_1, b_2$ is $D( \mu ) \in \cD_c (2)$. To verify the first identity in (\ref{eqn:48a}), fix $n \in \bN$ and 
$i_1, \ldots , i_n \in \{ 1,2 \}$, and calculate directly
\begin{align*}
\cf_{ (i_1, \ldots , i_n) } ( M_{D( \mu )} )
& = \varphi ( b_{i_1} \cdots b_{i_n} )                   \\
& = \varphi \Bigl( (t_{i_1,1} a^1 + t_{i_1,*} a^{*} )
\cdots (t_{i_n,1} a^1 + t_{i_n,*} a^{*} )  \Bigr)        \\
& = \sum_{ \ell_1, \ldots , \ell_n \in \{ 1,* \} } \
     t_{i_1, \ell_1} \cdots t_{i_n, \ell_n} \
    \varphi ( a^{\ell_1} \cdots a^{\ell_n} )             \\
& = \sum_{ \ell_1, \ldots , \ell_n \in \{ 1,* \} } \
     t_{i_1, \ell_1} \cdots t_{i_n, \ell_n} 
     \cf_{ ( \ell_1 , \ldots , \ell_n ) } (M_{\mu})   \\
& = \cf_{ (i_1, \ldots , i_n) } \widetilde{D} ( M_{\mu} ).
\end{align*}
The second equality in (\ref{eqn:48a})
follows from a similar multilinearity argument, using the fact that the $C^{*}$-probability space 
$( \cA , \varphi )$ carries a family of multilinear 
functionals $( \kappa_n : \cA^n \to \bC )_{n=1}^{\infty}$, 
called free cumulant functionals, such that
\[
   \begin{array}{l}
\cf_{ (i_1, \ldots , i_n) } ( R_{D( \mu )} )
= \kappa_n ( b_{i_1}, \ldots , b_{i_n} ),   
\quad n \in \bN ,
\, i_1, \ldots , i_n \in \{ 1,2 \} ,   \mbox{ and }   \\
                                                        \\
\cf_{ ( \ell_1 , \ldots , \ell_n ) } (R_{\mu})  
= \kappa_n ( a^{\ell_1}, \ldots , a^{\ell_n} ),   
\quad n \in \bN ,
\, \ell_1, \ldots , \ell_n \in \{ 1,* \} . 
\end{array}   
\]
(This multilinearity argument is precisely the one used
to describe the behavior of the $R$-transform under linear 
transformations  \cite[Proposition 16.12]{NS2006}.)

The third equality (\ref{eqn:48a})
follows from a similar multilinearity argument, using the Boolean
cumulant functionals $( \beta_n : \cA^n \to \bC )_{n=1}^{\infty}$ (for 
a discussion of Boolean cumulants see, for instance,
\cite[Section 4.6]{L2004}).
\end{proof}

\begin{lemma}  \label{lemma:49}
Let $D : \cD_c (1,*) \to \cD_c (2)$ be the 
bijection defined in Proposition \emph{\ref{prop:46}}. Then:
\begin{enumerate}

\item[\rm(1)] $D( \mu \boxplus \mu ' ) = D( \mu  ) \boxplus D ( \mu ' )$ for every $\mu, \mu ' \in \cD_c (1,*)$.

\item[\rm(2)] $D( \Dcinfdiv (1,*)) = \Dcinfdiv (2)$.
\end{enumerate}
\end{lemma} 

\begin{proof} (1) Since every $\lambda \in \cD_c (2)$ is 
uniquely determined by its $R$-transform, it suffices to 
verify that $D( \mu \boxplus \mu ' )$ and 
$D( \mu  ) \boxplus D ( \mu ' )$ have the same $R$-transform.  
Indeed,
\begin{align*}
R_{D( \mu \boxplus \mu ')}
& = \widetilde{D} ( R_{\mu \boxplus \mu '} )    
    \mbox{ (by Lemma \ref{lemma:48})}                        \\
& = \widetilde{D} ( R_{\mu} ) + \widetilde{D} ( R_{ \mu '} ) 
    \mbox{ (since $R_{\mu \boxplus \mu '} = R_{\mu} + R_{\mu '}$ 
           and $\widetilde{D}$ is linear) }                    \\
& = R_{ D( \mu ) } + R_{ D( \mu ') }
    \mbox{ (by Lemma \ref{lemma:48})}                         \\
& = R_{ D( \mu ) \boxplus D( \mu ') } .
\end{align*}

Part (2) follows immediately from (1) and from the definition of $\boxplus$-infinite divisibility.
\end{proof}

\begin{proof}  [Proof of Theorem \emph{\ref{thm:42}}]
We define the required bijection $\bB_{(1,*)}$ so that the diagram
\[
\begin{matrix}
\Dc (1,*) & \stackrel{\bB_{(1,*)}  }{{\mbox{\huge{$\longrightarrow{}$}}}} & 
\Dcinfdiv (1,*) \cr
& & & \cr
{D}{\mbox{\huge{$\downarrow$}}} & & 
{{\mbox{\huge{$\downarrow$}}}{D}} \cr
& & & \cr
\Dc(2)   & \stackrel{\bB_2  
}{{\mbox{\huge{$\longrightarrow{}$}}}} & \Dcinfdiv(2)
\end{matrix}
\]
is commutative, where $D$ is defined in Proposition \ref{prop:46} and
 $\bB_2$ is provided by Theorem \ref{thm:44} for $k=2$.  More precisely, let $D_0 : \Dcinfdiv (1,*) \to \Dcinfdiv (2)$ be the restriction of $D$; this is a bijection by Lemma \ref{lemma:49}(2)). Then define
\[
\bB_{(1,*)} := D_0^{-1} \circ \bB_2 \circ D.
\]
Pick an arbitrary $\mu \in \cD_c (1,*)$, denote 
$\bB_{(1,*)} ( \mu ) = \nu$. We prove that 
$R_{\nu} = \eta_{\mu}$.  Since $\widetilde{D}$ is injective, 
it suffices to verify that 
$\widetilde{D} ( R_{\nu} ) = \widetilde{D} (\eta_{\mu} )$.
Indeed, the definition of $\bB_{(1,*)}$ implies 
$\bB_2 (  D( \mu )  ) = D( \nu )$, and the definition 
of $\bB_2$ yields $R_{ D( \nu ) } = \eta_{ D( \mu ) }$. 
Thus 
$$
\widetilde{D} ( R_{\nu} ) = R_{D ( \nu )} 
= \eta_{ D( \mu ) } 
= \widetilde{D} ( \eta_{\mu} ) ,
$$
as required.
\end{proof}

\section{Parametrization of $\eta$-diagonal 
distributions in $\cD_c (1,*)$}

We show that an $\eta$-diagonal 
$*$-distribution $\mu \in \cD_c (1,*)$ is naturally
parametrized by the pair of compactly supported 
probability measures on $[0, \infty )$ that arise as the
distributions of $ZZ^{*}$ and of $Z^{*}Z$ in the 
noncommutative probability space
$( \bC \langle Z, Z^{*} \rangle, \mu )$. 

\begin{defrem}   \label{def:51}
Suppose that  $\mu\in\cD_c (1,*)$ is the $*$-distribution of an element $a$ in a 
$C^{*}$-probability space $( \cA , \varphi )$. Basic considerations on positive elements in a 
$C^{*}$-probability space (see, for instance,  \cite[Propositions 
3.13 and 3.6]{NS2006}) show the existence of 
compactly supported Borel probability measures $\sigma_1, \sigma_2$  
on $[ 0, \infty )$ such that 
\[
\varphi (  (aa^{*})^n  ) 
= \int_0^{\infty} t^n \, d \sigma_1 (t) \mbox{ and }
\varphi (  (a^{*}a)^n  )
= \int_0^{\infty} t^n \, d \sigma_2 (t), \quad n \in \bN .
\]
Thus $\sigma_1$ and $\sigma_2$ satisfy
\begin{equation}   \label{eqn:51a}
\int_0^\infty t^n\, d\sigma_1(t) = \mu((ZZ^*)^n),
\quad n \in \bN
\end{equation}
and
\begin{equation}   \label{eqn:51b}
\int_0^\infty t^n\, d\sigma_2(t) = \mu((Z^*Z)^n) 
\quad n \in \bN .
\end{equation} 
Moreover, $\sigma_1$ and $\sigma_2$ are uniquely 
determined by (\ref{eqn:51a}) and (\ref{eqn:51b}) since a 
compactly supported probability measure on $\bR$ is determined
by its moments.  We refer to  $\sigma_1$ and $\sigma_2$ as 
 the {\em distributions} of $ZZ^{*}$ and 
 $Z^{*}Z$, respectively, in the $*$-probability space 
$( \bC \langle Z, Z^{*} \rangle , \mu )$.
\end{defrem}

The following theorem provides the parametrization announced in the title of the section.

\begin{theorem}   \label{thm:51}
Let $\cPplus_c$ denote the set of all compactly supported Borel 
probability measures on $[ 0, \infty )$.  There is bijective 
map
\[
\Phi: \cPplus_c \times \cPplus_c \to
\{ \mu \in \cD_c (1,*) : \mu \mbox{ is $\eta$-diagonal} \} 
\]
described as follows: given $\sigma_1, \sigma_2 \in \cP^+_c$, 
 $\Phi ( \sigma_1, \sigma_2 )$ is the unique $\eta$-diagonal
$*$-distribution $\mu \in \cD_c (1,*)$ 
such that the distributions of $ZZ^{*}$ and $Z^{*}Z$ in 
$( \bC \langle Z, Z^{*} \rangle , \mu )$
are equal to $\sigma_1$ and $\sigma_2$, respectively.
\end{theorem}

The point of Theorem \ref{thm:51} is that the map $\Phi$ 
is defined on all of $\cPplus_c \times \cPplus_c$. In other words, for
every $\sigma_1, \sigma_2 \in \cPplus_c$ there exists an 
$\eta$-diagonal $*$-distribution $\mu \in \cD_c (1,*)$ such 
that  (\ref{eqn:51a}) and (\ref{eqn:51b}) hold.  We  
prove this by
producing an {\em operator model} 
for $\mu$:  starting from $\sigma_1$ and $\sigma_2$ we  
construct explicitly an operator $A$ on a Hilbert space $\cK$
such that the $*$-distribution of $A$ with respect to a
suitably chosen functional on $B( \cK )$ is the required $\eta$-diagonal
distribution.  The bulk of this section is 
devoted to the description of the operator model. At the end, we complete the proof of Theorem \ref{thm:51}.  The construction of the operator 
model is described in the next remark.

\begin{remnotation}   \label{rem:52}
({\em Description of the operator model.})
Fix $\sigma_1,\sigma_2\in\cPplus_c$ which we take as the input for our construction of an $\eta$-diagonal operator.
In the description of the construction, it is convenient to 
 use the \emph{symmetric square roots} of $\sigma_1$ and 
$\sigma_2$.  These are the symmetric compactly supported Borel 
probability measures $\widetilde{\sigma}_1$ and 
$\widetilde{\sigma}_2$ on $\bR$ with moments given by 
the formula
\[
\int_{-\infty}^{\infty} t^n \, d \widetilde{\sigma_j} (t) =
\begin{cases} 0,&n\text{ odd}  \\
\int_0^{\infty} t^{n/2} \, d \sigma_j (t) ,&n\text{ even,
}
\end{cases} 
\]
for $j = 1,2$. Our construction of an $\eta$-diagonal operator proceeds in three steps.

\noindent{\bf Step 1.} We construct a Hilbert space $\cH$, an operator $X\in\cB(\cH)$, and vectors $\xi_1,\xi_2\in\cH$  with the following properties:
\begin{enumerate}
	\item[(1a)] $||\xi_1||=||{\xi}_2||=1$,
	\item[(1b)] $\langle \xi_1, \xi_2 \rangle=0$,
	\item[(1c)] $\langle X^k\xi_1,{\xi}_2\rangle=\langle X^k{\xi}_2,\xi_1 \rangle=0$ for $k\in\mathbb{N}$,
	\item [(1d)]$\langle X^{2k-1}\xi_j,\xi_j\rangle=0$ and $\langle X^{2k}\xi_j,\xi_j\rangle=\int_{0}^{\infty} t^{k} \, d \sigma_j (t)$ for $k\in\mathbb N$ and $j=1,2.$
\end{enumerate}
In other words, property (1d) says that $X$ has distribution 
$\widetilde{\sigma}_1$ with respect to the vector state 
defined by $\xi_1$, and distribution $\widetilde{\sigma}_2$ with 
respect to the one defined by $\xi_2$, on the operator algebra
 $B( \cH )$. For the actual construction of $X$ consider, 
for $j=1,2$, Hilbert spaces $\cM_j$, operators $T_j\in\cB(\cM_j)$,
and unit vectors $\eta_j\in\cM_j$, such that the distribution of 
$T_j$ with respect to the vector state $\eta_j$ is 
$\widetilde{\sigma}_{j}$.  Then set $\cH=\cM_1\oplus{\cM}_2$, 
$X=T_1\oplus{T}_2$,  $\xi_1=\eta_1 \oplus 0$, and 
${\xi}_2=0\oplus{\eta_2}$. Properties (1a)--(1d) are then easily 
verified.  (In subsequent steps we only use the properties 
(1a)--(1d). The precise description of  $\mathcal{H},X,\xi_1$, 
and ${\xi}_2$ is not necessary.)

\noindent{\bf Step 2.} Define a rank-one partial isometry $Y\in\cB(\cH)$ by setting $Y (\zeta)=\langle \zeta , {\xi}_1 \rangle\xi_2$ for  $\zeta\in \cH$. We have $Y{\xi}_1=\xi_2$,   $Y^*\zeta=\langle \zeta , \xi_2 \rangle{\xi}_1$ for $\zeta\in\cH$, and
\begin{equation}   \label{eqn:52a}
YY^*\zeta  = \langle \zeta , \xi_2 \rangle\xi_2,
\ \ Y^*Y\zeta  = \langle \zeta , {\xi}_1 \rangle{\xi}_1 .
\end{equation}
Thus $YY^{*}$ and $Y^{*}Y$ are the orthogonal projections onto 
the $1$-dimensional spaces generated by 
$\xi_2$ and ${\xi}_1$, respectively.

\noindent{\bf Step 3.} Consider the Hilbert space 
$\mathcal{K}=\cH\otimes\cH$ and the unit vector 
$\xi=\xi_1 \otimes{\xi}_2\in\mathcal{K}$, then consider the
$C^{*}$-probability space $(\cB(\mathcal{K}),\varphi_\xi)$, 
where $\varphi_\xi(T)=\langle T\xi,\xi\rangle$ for 
$T\in\cB(\mathcal{K})$.  Let $V \in B( \cH \otimes \cH )$ be the
flip operator determined by the requirement that
by $V(\zeta\otimes\zeta^\prime)=\zeta^\prime\otimes\zeta$, $\zeta,\zeta'\in\mathcal H$.
Note that $V$ is a symmetry (that is, it is self-adjoint and  
$V^2=I$).  Finally, define  $A=V(Y \otimes X)$.

This concludes the construction of the variable $A$ in 
$( B ( \mathcal{K} ),\varphi_\xi)$. 
\end{remnotation}

We now take on the proof that the operator $A$ constructed above has
the desired $\eta$-diagonal distribution with respect to the 
functional $\varphi_{\xi}$.  We start by recording 
some easily verified identities satisfied by $A$, the proof of
which is left to the reader.

\begin{lemma}    \label{lemma:53}
Consider the framework of Remark \emph{\ref{rem:52}}.  We have
\begin{equation}   \label{eqn:53a}
AA^* = X^2 \otimes YY^*, \ \ A^*A = Y^*Y \otimes X^2,
\end{equation}
and 
\begin{equation}   \label{eqn:53b}
A^2 = XY \otimes YX, \ \ \left(A^*\right)^2 = Y^*X \otimes XY^*.
\end{equation}
\hfill $\square$
\end{lemma}

The following lemma establishes the distributions of $AA^*$ and $A^*A$ along with a few non-alternating $*$-moments of $A$.

\begin{lemma}   \label{lemma:54}
Let $A$ be as above, then for any integer $k\geq 0$ we have
\begin{enumerate}
\item[\rm(1)]	$\varphi_\xi((AA^*)^k)=
                 \int_0^{\infty} t^k \, d \sigma_1 (t)$,
\item[\rm(2)]	$\varphi_\xi(A(AA^*)^k)=  0$, 
\item[\rm(3)]	$\varphi_\xi(A^*(AA^*)^k)=  0$,
\item[\rm(4)]	$\varphi_\xi((A^*A)^k)=
                 \int_0^{\infty} t^k \, d \sigma_2 (t)$,
\item[\rm(5)]	$\varphi_\xi(A(A^*A)^k)=  0$,
\item[\rm(6)]	$\varphi_\xi(A^*(A^*A)^k)=  0$.
\end{enumerate}
\end{lemma}

\begin{proof}
We verify only the first three equations. The proof of (4)--(6) is similar. We have
	$$(AA^*)^k=X^{2k}\otimes(YY^*)^k=X^{2k}\otimes(YY^*),$$ so
	$$\varphi_\xi((AA^*)^k)=\langle (X^{2k}\otimes(YY^*))\xi,\xi\rangle=\langle X^{2k}\xi_1,\xi_1\rangle,$$
and (1) follows from property (1d) in Step 1 of the construction of 
$A$. To prove (2), we calculate
	$$ A(AA^*)^k=V(Y\otimes X)(X^{2k}\otimes YY^*),$$
	thus
	\begin{align*}
	\varphi_\xi(A(AA^*)^k)&=\langle (Y\otimes X)(X^{2k}\otimes YY^*){\xi}_1\otimes{\xi_2},V{\xi}_1\otimes{\xi}_2\rangle\\
	&=\langle(Y\otimes X)(X^{2k}{\xi_1}\otimes{\xi}_2 ),{\xi}_2\otimes{\xi}_1\rangle\\
	&=\langle YX^{2k}{\xi_1},{\xi}_2\rangle\langle X{\xi}_2,\xi_1\rangle=0,
	\end{align*}
	because $\langle   X {\xi}_2 , \xi_1 \rangle=0$, thereby concluding the proof of (2).
	Similarly,
	$$A^*(AA^*)^k=(X\otimes Y^*) V ( X^{2k} \otimes YY^* ),$$ so
	\begin{align*}
	\varphi_\xi(A^*(AA^*)^k)=&\langle ((X\otimes Y^*) V ( X^{2k} \otimes YY^* )\xi,\xi\rangle=\langle X\otimes Y^*(\xi_2\otimes X^{2k}\xi_1),\xi_1\otimes \xi_2\rangle\\=&\langle  X\xi_2,{\xi}_1\rangle\langle Y^*X^{2k}{\xi}_1,\xi_2\rangle=0
	\end{align*}
by property (1c) in Step 1 of the construction of $A$.
\end{proof}

The next lemma gives some  properties of the operator $A$ that 
are useful in verifying its  $\eta$--diagonality.

\begin{lemma}
\label{lem:42}
Let $A$ be as above, then for any integer $k\geq 0$ we have
\begin{enumerate}
\item[\rm(1)]$A^2(AA^*)^k\xi=0$,
\item[\rm(2)]$(A^*)^2(AA^*)^k\xi=0$, 
\item[\rm(3)]$(A^*A)(AA^*)^k\xi 
    = \int_0^{\infty} t^k \, d \sigma_1 (t) \cdot A^*A \xi$,
\item[\rm(4)]$A^2(A^*A)^k\xi=0$,
\item[\rm(5)]$(A^*)^2(A^*A)^k{\xi}_=0$,
\item[\rm(6)]$(AA^*)(A^*A)^k{\xi}
    = \int_0^{\infty} t^k \, d \sigma_2 (t) \cdot AA^* \xi$.
\end{enumerate}
\end{lemma}
\begin{proof}
As in the previous proof we only verify (1)--(3). We have	
	$$(A)^2(AA^*)^k=(XY\otimes YX)(X^{2k}\otimes YY^*)=XYX^{2k}\otimes YXYY^*,$$
	and using the fact that $YY^*{\xi}_2={\xi}_2$ we see that
	$	YXYY^*{\xi}_2=\langle X\xi_2,{\xi}_1\rangle\xi_2=0$ by property (1c).
	Similarly,
	$$
	(A^*)^2(AA^*)^k=(Y^*X\otimes XY^*)(X^{2k}\otimes YY^*)=Y^*X^{2k+1}\otimes XY^*,
	$$
	and (2) follows because
	$Y^*X^{2k+1}\xi_1=\langle X^{2k+1} \xi_1, {\xi}_2 \rangle \xi_1=0$
	by  (1c).
	Finally, (\ref{eqn:53a}) yields
	\[
	(A^*A)(AA^*)^k=(Y^*Y\otimes X^2) (X^{2k}\otimes YY^*)=Y^*YX^{2k}\otimes X^2 YY^*.
	\]
Observe that $Y^*YX^{2k}\xi_1
=\langle X^{2k}\xi_1 , \xi_1 \rangle \xi_1
= \int_0^{\infty} t^k \, d \sigma_1 (t) \cdot \xi_1$ by (1d), 
while $YY^*{\xi}_2={\xi}_2$.  Therefore
\[
(A^*A)(AA^*)^k
= \int_0^{\infty} t^k \, d \sigma_1 (t) \cdot
(\xi_1\otimes X^2 {\xi}_2)
= \int_0^{\infty} t^k \, d \sigma_1 (t) \cdot A^*A \xi ,
\]
thus proving (3).
\end{proof}

\begin{corollary}\label{mixt-alt} 
Let $W=W_1W_2\cdots W_d$ be a mixed-alternating word in $A$ and $A^*$, factored as in 
	\emph{\eqref{eqn:27b}} with $d\ge2$. Then
	$$W\xi=\varphi_\xi(W_2)\varphi_\xi(W_3)\cdots\varphi_\xi(W_d)W_1\xi.$$
\end{corollary}
\begin{proof}
	Parts (3) and (6) of the preceding lemma yield the conclusion when $d=2$.  The general case follows easily by induction on $d$.
\end{proof}

\begin{proposition}  \label{prop:56}
Let $\sigma_1$ and $\sigma_2$ be probability measures in 
$\cPplus_c$, and let the operator $A$ in 
$(\mathcal{B}(\mathcal{K}),\varphi_\xi)$ be constructed as
in Remark \emph{\ref{rem:52}}.  Then the $*$-distribution of
$A$ is $\eta$-diagonal.  Moreover, the distributions of  
$AA^*$ and $A^*A$ are  $\sigma_1$ and ${\sigma}_2$, respectively.
\end{proposition}

\begin{proof}
	The second assertion follows from parts (1) and (4) of Lemma  \ref{lemma:54}.
	It remains  to prove that the distribution of $A$ is  $\eta$-diagonal, and to do this we verify the conditions in 
Theorem \ref{thm:28}.  Let $W=W_1W_2\cdots W_d$ be a mixed-alternating word in $A$ and $A^*$, factored as in 
	\emph{\eqref{eqn:27b}}. Corollary \ref{mixt-alt} yields
	$$\varphi_\xi(W)=\varphi_\xi(W_2)\varphi_\xi(W_3)\cdots\varphi_\xi(W_d)\langle W_1\xi,\xi\rangle=\prod_{j=1}^d\varphi_\xi(W_j),$$
	thus verifying condition ($\eta$DM2). Finally we verify condition ($\eta$DM1). Suppose that  $V$ is a word in $A$ and $A^*$ that is not mixed-alternating, and choose a mixed-alternating word $W$ of maximum length with the property that $V$ can be written as $V=UW$ for some non-empty word $U$. Also, write $W=W_1W_2\cdots W_d$ as in 
	\emph{\eqref{eqn:27b}}. We have  
	$$
	\varphi_\xi(V)=\langle V\xi,\xi\rangle=\prod_{j=2}^d\varphi_\xi(W_j)\langle UW_1\xi,\xi\rangle
	$$
	by Corollary \ref{mixt-alt}. If $|U|=1$, the equality
	$\varphi_\xi(V)=0$ follows from Lemma \ref{lemma:54}. If $|U|\ge2$, then $U$ is of the form $U'AA$ or $U'A^*A^*$ for some (possibly empty) word $U'$.  In this case,  $\varphi_\xi(V)=0$ by Lemma \ref{lem:42}.
\end{proof}

We conclude the discussion of the parametrization 
announced at the beginning of the section.

\begin{proof}  [Proof of Theorem \emph{\ref{thm:51}}]

We first note that the map $\Phi$ is well-defined.  Indeed, let 
$\sigma_1, \sigma_2 \in \cPplus_c$ be given.  The existence of an
$\eta$-diagonal $*$-distribution $\mu \in \cD_c (1,*)$ which 
fulfils the conditions (\ref{eqn:51b}) is ensured by Proposition 
\ref{prop:56}.  The uniqueness of $\mu$ follows from the fact
that an $\eta$-diagonal $*$-distribution is completely 
determined by its alternating $*$-moments, as we saw in 
Theorem \ref{thm:28}.

The surjectivity of $\Phi$ is immediate from its definition: 
every $\eta$-diagonal $*$-distribution $\mu \in \cD_c (1,*)$ can 
be written as $\Phi ( \sigma_1, \sigma_2 )$, where 
$\sigma_1 , \sigma_2 \in \cPplus_c$ are the distributions of 
$ZZ^{*}$ and respectively $Z^{*}Z$ in 
$( \bC \langle Z, Z^{*} \rangle , \mu )$, in the sense discussed
in Definition \ref{def:51}.

Finally, the injectivity of $\Phi$ is immediate as well.  Indeed, 
if $\Phi ( \sigma_1 , \sigma_2 ) = \mu$, then the moments of
$\sigma_1$ and $\sigma_2$ can be retrieved as alternating moments
of $\mu$, and compactly supported probability measures on $\bR$
are determined by their moments.
\end{proof}

\section{ Parametrization of infinitely divisble
$R$-diagonal distributions}

In this section we use the BBP method to characterize
$\boxplus$-infinitely divisible $R$-diagonal distributions.
The parametrization mentioned in the 
title of the section arises naturally, in the way indicated in the 
following remark.

\begin{remnotation}   \label{rem:61}
Let $\Rcinfdiv$ denote the set of all the $R$-diagonal 
distributions in $\cD_c (1,*)$ that are $\boxplus$-infinitely
divisible.  It is immediate that the 
bijection $\bB_{(1,*)}$ from Theorem \ref{thm:42} induces 
a bijection (still denoted $\bB_{(1,*)}$)
\begin{equation}   \label{eqn:61a}
\bB_{(1,*)} : \{ \mu \in \cD_c (1,*) :
\mu \mbox{ is $\eta$-diagonal} \} \to \cR_c^{\mathrm{(inf-div)}}.
\end{equation}
On the other hand,  Theorem \ref{thm:51} provides 
a natural bijection 
\begin{equation}   \label{eqn:61b}
\Phi: \cPplus_c \times \cPplus_c \to
\{ \mu \in \cD_c (1,*) : \mu \mbox{ is $\eta$-diagonal} \}. 
\end{equation}
The map
\begin{equation}   \label{eqn:61c}
\Psi := \bB_{(1,*)} \circ \Phi
: \cPplus_c \times \cPplus_c 
\to \cR_c^{\mathrm{(inf-div)}}
\end{equation}
is therefore a bijection as well.  We refer to  $\Psi$ 
as the {\em BBP parametrization of $\Rcinfdiv$}. Every choice of parameters $\sigma_1, \sigma_2 \in \cPplus_c$
yields a distribution 
$\nu = \Psi ( \sigma_1, \sigma_2 ) \in \Rcinfdiv$ , and every
$\nu \in \Rcinfdiv$ arises from a unique pair $\sigma_1, \sigma_2$.

We emphasize that the bijection $\Psi$ works in a really 
straightforward way -- the coefficients of the $\eta$-series 
of $\sigma_1$ and $\sigma_2$ give the determining sequences 
of $\nu = \Psi ( \sigma_1 , \sigma_2 )$.
It is actually worth recording a direct consequence of 
this fact, as follows.
\end{remnotation}

\begin{notation}  \label{def:62a} 
We denote by $\cEplus_c$ the collection of those series 
$f \in \bC [[z]]$ with the property that $f = \eta_{\sigma}$ for 
some $\sigma \in \cPplus_c $ (where $\sigma$ is, a fortiori, 
uniquely determined).
\end{notation}

\begin{proposition}   \label{prop:62b}
Let $\nu \in \cD_c (1,*)$ be an $R$-diagonal 
distribution, and let $( \alpha_n )_{n=1}^{\infty}$
and $( \beta_n )_{n=1}^{\infty}$ be its determining 
sequences. Set   
$f(z)  = \sum_{n=1}^{\infty} \alpha_n z^n$ and
$g(z)  = \sum_{n=1}^{\infty} \beta_n z^n$. Then 
$\nu$ is $\boxplus$-infinitely divisible if and only if 
both $f$ and $g$ belong to $\cEplus_c$. 
\end{proposition}

\begin{proof}  If $\nu \in \Rcinfdiv$, then 
$\nu = \Psi ( \sigma_1 , \sigma_2 )$ for some 
$\sigma_1 , \sigma_2 \in \cPplus_c$, hence 
$f = \eta_{\sigma_1}$ and $g = \eta_{\sigma_2}$, and so 
$f,g \in \cEplus_c$.
Conversely, suppose that $f,g \in \cEplus_c$, so
$f = \eta_{\sigma_1}$ and $g = \eta_{\sigma_2}$ for 
some $\sigma_1 , \sigma_2 \in \cPplus_c$.  Then the distribution 
$\widetilde{\nu} := \Psi ( \sigma_1 , \sigma_2 )$ belongs to
$\Rcinfdiv$, and the definition of $\Psi$ shows that 
$\widetilde{\nu}$ has the same determining sequences as $\nu$.  
This forces $\nu = \widetilde{\nu}$, hence $\nu \in \Rcinfdiv$.
\end{proof}

The criterion provided by Proposition \ref{prop:62b} is 
useful because one can (following the work in \cite{BB2005})
characterize the series from $\cEplus_c$ in terms of the associated 
analytic functions.  We will follow up on this in the application 
presented in Section 7.

Since we are dealing with free probabilistic structures, it is 
natural to ask what is the description of the BBP parametrization 
$\Psi$ in terms of $R$-transforms.  Recall (Remark \ref{rem:34}) 
that an $R$-diagonal $*$-distribution $\nu \in \cD_c ( 1,*)$ is 
uniquely determined by the $R$-transforms
$R_{ {ZZ^{*}} }, R_{ {Z^{*}Z} } \in \bC [[z]]$.
The following result thus provides  an alternative 
characterization of what is $\Psi ( \sigma_1, \sigma_2 )$.

\begin{theorem}   \label{thm:62}
Let $\sigma_1, \sigma_2\in\cPplus_c$ and set 
$\nu = \Psi ( \sigma_1, \sigma_2 )$. Then the $R$-transforms of 
$ZZ^{*}$ and of $Z^{*}Z$ in the $*$-probability
space $( \bC \langle Z, Z^{*} \rangle, \nu )$ are 
described as follows:
\begin{equation}   \label{eqn:62a}
R_{ {ZZ^{*}} } (z) 
= R_{ \bB ( \sigma_1 ) }
\left(  z( 1 + M_{ \bB ( \sigma_2 ) } (z))  \right) 
,\quad  
R_{ {Z^{*}Z} } (z) 
= R_{ \bB ( \sigma_2 ) }
\left( z( 1 + M_{ \bB ( \sigma_1 ) } (z)) \right),
\end{equation}
where $\bB(\sigma_1)$ and $\bB(\sigma_2)$ indicate
the original  BBP bijection \emph{(}as discussed in 
Remark \emph{\ref{rem:45})}.
\end{theorem}

\begin{proof}  
We set $\mu := \Phi ( \sigma_1, \sigma_2 )$, so $\nu$ is 
$R$-diagonal, $\mu$ is $\eta$-diagonal, and 
$\bB_{(1,*)} ( \mu ) = \nu$.  Thus
\[
R_{\nu} (z,z^{*})
= \eta_{\mu} (z,z^{*})
= \sum_{n=1}^{\infty} \alpha_n (zz^{*})^n
+ \sum_{n=1}^{\infty} \beta_n (z^{*}z)^n ,
\]
where $( \alpha_n )_{n=1}^{\infty}$ and $( \beta_n )_{n=1}^{\infty}$
are the (common) determining sequences for $\mu$ and for $\nu$.
By the definition of the bijection $\Phi$ in Theorem \ref{thm:51},
$\sigma_1$ has the same moments as the element $ZZ^{*}$ in the 
noncommutative probability space 
$( \bC \langle Z, Z^{*} \rangle , \mu )$.  This implies that 
$\eta_{\sigma_1} = \eta_{_ {ZZ^{*}} }$, and then Proposition 
\ref{prop:212} gives us the formula
$\eta_{\sigma_1} (z) = \sum_{n=1}^{\infty} \alpha_n z^n$.
In a similar way we find that 
$\eta_{\sigma_2} (z) = \sum_{n=1}^{\infty} \beta_n z^n$.

Consider now the probability measures 
$\bB ( \sigma_1 ), \bB ( \sigma_2 ) \in \cP_c$.  The 
definition of  $\bB$ implies
\[
R_{\bB ( \sigma_1 )} (z) 
= \eta_{\sigma_1} (z) = \sum_{n=1}^{\infty} \alpha_n z^n ,
\ \ 
R_{\bB ( \sigma_2 )} (z) 
= \eta_{\sigma_2} (z) = \sum_{n=1}^{\infty} \beta_n z^n .
\]
But then Proposition \ref{prop:36} applies to the $R$-diagonal
$*$-distribution $\nu$ and yields (\ref{eqn:62a}).
\end{proof}

As a consequence of Theorem \ref{thm:62}, we obtain a 
natural connection between the notions of $\boxplus$-infinite 
divisibility in $\cD_c (1,*)$ and in $\cP_c$.  This is stated
in the next corollary.  The converse of the corollary fails 
even in the tracial framework (see Remark \ref{rem:67} below).

\begin{corollary}   \label{cor:63}
Let $\nu\in\cD_c (1,*)$ be $R$-diagonal 
and let $\tau_1, \tau_2 \in \cPplus_c$ be the distributions 
of $ZZ^{*}$ and $Z^{*}Z$ in the $*$-probability space 
$( \bC \langle Z, Z^{*} \rangle , \nu )$ \emph{(}as discussed in 
Definition \emph{\ref{def:51})}.  If $\nu$ is 
$\boxplus$-infinitely divisible in $\cD_c (1,*)$, then  $\tau_1$ and $\tau_2$ are 
$\boxplus$-infinitely divisible in $\cP_c$.
\end{corollary}

\begin{proof}  By symmetry, it suffices to show that  $\tau_1$ is $\boxplus$-infinitely divisible.
According to \cite[Theorem 4.3]{V1986},
a compactly supported Borel probability measure on 
$\bR$ is $\boxplus$-infinitely divisible if and only if its 
$R$-transform can be extended to an analytic self-map of the 
upper half-plane $\bC^{+}$. Suppose that $\nu=\Phi(\sigma_1,\sigma_2)$, where
 $\sigma_1, \sigma_2 \in \cPplus_c$.  The $R$-transform 
$R_{\tau_1}$, which is the same as $R_{ {ZZ^{*}} }$, is given 
by the first Equation (\ref{eqn:62a}).  By  \cite[Proposition 
6.1]{BV1993},  the moment series of the 
probability measure $\bB ( \sigma_2 )$ can be extended analytically 
to $\bC^{+}$ and this extension satisfies
\[
z \in \bC^{+} \Rightarrow
z ( 1 + M_{ \bB ( \sigma_2 )} (z) ) \in \bC^{+}.
\]
Finally, since $\bB ( \sigma_1 )$ is $\boxplus$-infinitely 
divisible, \cite[Theorem 4.3]{V1986} assures us 
that $R_{\bB ( \sigma_1 )}$ extends analytically to a self-map 
of $\bC^{+}$.  We conclude that for every $z  \in \bC^{+}$, 
$R_{\bB ( \sigma_1 )}$ is defined at 
$z ( 1 + M_{ \bB ( \sigma_2 )} (z) )$,  and that
\[
z \mapsto R_{\bB ( \sigma_1 )} 
\bigl( \, z ( 1 + M_{ \bB ( \sigma_2 )} ) (z) \, \bigr)
\]
is an analytic self-map on $\bC^{+}$, as required.
\end{proof}

In the remainder of this section, we discuss the KMS 
example.  In this special case one can process further the 
formulas from Theorem \ref{thm:62} and arrive at explicit
formulas (stated in Proposition \ref{prop:66}) for the 
distributions of $ZZ^{*}$ and of $Z^{*}Z$ 
in terms of the probability measures $\sigma_1, \sigma_2$ 
that parametrize $\nu$.  These formulas call
on some commonly used operations from the free harmonic 
analysis of $\cPplus_c$, that are reviewed in the following 
remark.

\begin{remark}   \label{rem:64}
({\em Some elements of free harmonic analysis on $\cPplus_c$.})
(1) Measures $\sigma\in\cP_c$ have 
{\em free additive convolution powers} with real exponent $t \in [ 1, \infty )$. More precisely, for every 
$\sigma \in \cP_c$ and $t \in [ 1, \infty )$, there exists a unique measure
$\tau \in \cP_c$ such that 
$R_{\tau} = t R_{\sigma}$ (see   \cite[pp. 228-231]{NS2006}).  This measure $\tau$ is denoted 
$\sigma^{\boxplus t}$. When $t$ is an integer, $\sigma^{\boxplus t}$
is simply the $t$-fold convolution 
$\sigma \boxplus \cdots \boxplus \sigma$.  The argument in 
\cite[pp. 228-231]{NS2006} also 
shows that $\sigma^{\boxplus t} \in \cPplus_c$ for all $t \in [1, \infty)$ if $\sigma \in \cPplus_c$.

The analogous result for Boolean convolution provides 
for every $\sigma \in \cP_c$ and $t \in ( 0, \infty )$ a
\emph{Boolean convolution power} $\sigma^{\uplus t}\in \cP_c$ such that 
$\eta_{\sigma^{\uplus t}} = t \eta_{\sigma}$ 
(see \cite[Theorem 3.6]{SW1997}). As in the free case, 
$\sigma^{\uplus t}\in \cPplus_c$ for every $t \in ( 0, \infty )$ 
if $\sigma \in \cPplus_c$ (see, for instance, the operator model 
constructed in \cite[Proposition 4.8]{BN2008}).

(2) The original BBP bijection 
$\bB : \cP_c \to \Pcinfdiv$ (Remark \ref{rem:45}) can be 
expressed using convolution powers, by the formula
\[
\bB ( \sigma ) 
= \left( \sigma^{\boxplus 2} \right)^{\uplus 1/2},
\quad \sigma \in \cP_c,
\]
which was proved in  \cite[Theorem 1.2]{BN2008b}).  The facts 
reviewed in (1) 
above imply that $\bB ( \sigma )\in\cPplus_c$ for every 
$\sigma \in \cPplus_c$. 

(3) \emph{Free multiplicative convolution} $\boxtimes$ is another binary operation defined on the set $\cPplus_c$. This operation corresponds to the product of free random variables.  Quite remarkably, 
$\bB|\cPplus_c$ was shown in \cite[Remark 3.9]{BN2008b} to be a homomorphism for $\boxtimes$, that is,
\[
\bB ( \sigma \boxtimes \sigma ' ) 
= \bB ( \sigma ) \boxtimes \bB ( \sigma ' ), \quad \sigma, \sigma ' \in \cPplus_c.
\]

(4) The free counterpart of the standard Poisson distribution is the
\emph{Marchenko-Pastur distribution} $\Pi_1$ (also known as the \emph{the free Poisson distribution}). This distribution is supported on the interval $[0,4]$ and it is Lebesgue absolutely continuous with density
\[
d \Pi_1 (t)/dt = \frac{1}{2 \pi} \sqrt{ (4-t)/t } ,
\quad 0 \leq t \leq 4.
\]
Its $R$-transform is
\[
R_{ {\Pi_1} } (z) = z/(1-z),  
\]
and a simple calculation using the definition of $\bB$
shows that 
\begin{equation}   \label{eqn:64a}
\Pi_1 = \bB \left( \frac{1}{2} ( \delta_0 + \delta_2 ) \right).
\end{equation}
A useful property of $\Pi_1$ is that it converts moment
series into $R$-transforms via the formula 
\begin{equation}   \label{eqn:64b}
R_{\sigma \boxtimes \Pi_1} = M_{\sigma}, \quad \sigma \in \cPplus_c.
\end{equation}
See, for instance, 
\cite[Propositions 17.2 and 17.4]{NS2006}.
\end{remark}

\begin{remark}    \label{rem:65}
Let $\sigma\in\cPplus_c$ and let $t >0$ be a real 
number.  The determining sequences $( \alpha_n )_{n=1}^{\infty}$ and
$( \beta_n )_{n=1}^{\infty}$ of the infinitely divisible $R$-diagonal 
$*$-distribution 
$\nu := \Psi ( \sigma^{\uplus t} , \sigma ) \in \Rcinfdiv$ satisfy
\[
\sum_{n=1}^{\infty} \beta_n z^n = \eta_{\sigma} (z),\quad
\sum_{n=1}^{\infty} \alpha_n z^n = 
\eta_{\sigma^{\uplus t}} (z) = t \eta_{\sigma} (z). 
\]
 Thus $\nu$ satisfies the KMS condition with 
parameter $t$:  $\alpha_n = t \beta_n$, 
$n \in \bN$ (Definition \ref{def:36}).
\end{remark}

\begin{proposition}   \label{prop:66}
With the notation of the preceding remark,  let 
$\tau_1, \tau_2 \in \cPplus_c$ be the distributions of 
$ZZ^{*}$ and of $Z^{*}Z$ in the noncommutative probability 
space $( \bC \langle Z, Z^{*} \rangle, \nu )$  \emph{(}as 
discussed in Definition \emph{\ref{def:51})}.  Then
\begin{equation}  \label{eqn:66a}
\tau_1 = \left(  \bB ( \sigma ) 
         \boxtimes \Pi_1  \right)^{\boxplus t}
\mbox{ and }
\tau_2 = \left(  \bB ( \sigma )^{\boxplus t} 
         \boxtimes \Pi_1 \right)^{\boxplus 1/t}.
\end{equation}
\end{proposition}

\begin{proof}  The two formulas  in (\ref{eqn:66a}) have 
similar proofs. We only verify the first one.
Since $\tau_1$ is the distribution of $ZZ^{*}$, we have
$R_{\tau_1} = R_{{ZZ^{*}} }$, and Theorem \ref{thm:62} yields
\begin{equation}   \label{eqn:66b}
R_{\tau_1} (z) = R_{ \bB ( \sigma_1 ) }
\left(  z( 1 + M_{ \bB ( \sigma_2 ) } (z)) \right) ,
\end{equation}
where $\sigma_1 = \sigma^{\uplus t}$ and 
$\sigma_2 = \sigma$.
The relation $\sigma_1 = \sigma^{\uplus t}$ and  (\ref{eqn:1a}) imply
$\bB ( \sigma_1 ) = \bB ( \sigma )^{\boxplus t}$, and hence
$R_{\bB ( \sigma_1 )}  = t \, R_{\bB ( \sigma )}$.  The equality 
(\ref{eqn:66b})  can be continued as follows:
\begin{align*}
R_{\tau_1} (z)
& = t \cdot R_{\bB ( \sigma )} \left(
    z( 1 + M_{\bB ( \sigma )} (z) ) \right)                  \\
& = t \cdot M_{\bB ( \sigma )} (z)              
    \mbox{ (by (\ref{eqn:32e})) }                             \\
& = t \cdot R_{\bB ( \sigma ) \boxtimes \Pi_1} (z)         
    \mbox{ (by (\ref{eqn:64b})) }   \\
& = R_{( \bB ( \sigma ) \boxtimes \Pi_1)^{\boxplus t}} (z).
\end{align*}
Thus the probability measures $\tau_1$ and 
$( \bB ( \sigma ) \boxtimes \Pi_1)^{\boxplus t}$ are equal because they have the same $R$-transform.
\end{proof}

\begin{remark}   \label{rem:67}
({\em Tracial case.})
In the special case when $t = 1$,  the preceding proposition reduces to
\begin{equation}   \label{eqn:67a}
\tau_1  = \tau_2 = \bB ( \sigma ) \boxtimes \Pi_1 .
\end{equation}
Using (\ref{eqn:64a})
and  invoking the multiplicativity of $\mathbb B$ (Remark \ref{rem:64}(3)), we can rewrite 
(\ref{eqn:67a}) as
\begin{equation}   \label{eqn:67b}
\tau_1  = \tau_2 = \bB \left( \sigma \boxtimes 
\frac{1}{2} ( \delta_0 + \delta_2 ) \right).
\end{equation}
This confirms the fact (Corollary \ref{cor:63})
that $\tau_1$ and $\tau_2$ are $\boxplus$-infinitely divisible
in $\cP_c$.

We conclude with an argument showing that the converse of 
Corollary \ref{cor:63} does not hold.
Choose a distribution
$\widetilde{\sigma} \in \cPplus_c$ that cannot be written as 
$\sigma \boxtimes \frac{1}{2} ( \delta_0 + \delta_2 )$ for any $\sigma\in\cPplus_c$. (For instance, $\widetilde{\sigma} =\frac13(\delta_0+\delta_1+\delta_2)$ is such a distribution.) Let $\nu \in \cD_c (1,*)$
be the tracial $R$-diagonal $*$-distribution defined by the 
requirement that the common distribution of $ZZ^{*}$ and 
$Z^{*}Z$ in $( \bC \langle Z, Z^{*} \rangle, \nu )$ is equal 
to $\bB ( \, \widetilde{\sigma} \, )$ (see  \cite[Proposition 15.13]{NS2006} for an argument that $\nu$ exists). 
The distributions of $ZZ^{*}$ and $Z^{*}Z$ are 
$\boxplus$-infinitely divisible in $\cP_c$, by construction.  
We show that $\nu$ is not $\boxplus$-infinitely divisible in 
$\cD_c (1,*)$.  Suppose, to get a contradiction, that $\nu$ is $\boxplus$-infinitely divisible. Then $\nu=\Psi(\sigma,\sigma)$ for some 
$\sigma \in \cPplus_c$. Since 
$\tau_1 = \tau_2 = \bB (  \widetilde{\sigma}  )$,
 (\ref{eqn:67b}) yields
$\bB (  \widetilde{\sigma}  ) = \bB \left( \sigma 
\boxtimes \frac{1}{2} ( \delta_0 + \delta_2 ) \right)$, and thus
$\widetilde{\sigma} = \sigma 
\boxtimes \frac{1}{2} ( \delta_0 + \delta_2 )$ because 
$\mathbb B$ is injective, contrary to the choice of 
$\widetilde{\sigma}$.
\end{remark}

\begin{example}   \label{ex:68}
({\em $\lambda$-circular distribution.})
Let $\lambda > 0$ be a parameter.  If in the setting of 
Remark \ref{rem:65} and Proposition \ref{prop:66} we take 
$\sigma = \delta_1$ (Dirac mass at $1$) and $t = \lambda$, 
then the resulting $*$-distribution 
$\nu \in \Rcinfdiv$ is the $\lambda$-circular distribution 
mentioned right before Definition \ref{def:36}.  Indeed, it 
is immediate that in this case the series
$\sum_{n=1}^{\infty} \beta_n z^n$ and
$\sum_{n=1}^{\infty} \alpha_n z^n$ from Remark \ref{rem:65}
are reduced to $\eta_{ \delta_1 } (z) = z$ and 
respectively to $t \eta_{ \delta_1 } (z) = \lambda z$; 
hence we have $\alpha_1 = \lambda, \beta_1 = 1$ and 
$\alpha_n = \beta_n = 0$ for all $n \geq 2$, as required in 
the definition of the $\lambda$-circular distribution.

In this example, the formulas indicated in Proposition 
\ref{prop:66} for the distributions of $ZZ^{*}$ and of 
$Z^{*}Z$ give free Poisson distributions.  In order to 
make this precise, we need to review another bit of notation:
for any two parameters $p,q > 0$ one has a 
{\em free Poisson distribution of rate $p$ and jump size $q$},
which we will denote as $\Pi_{p;q}$, and which appears in the
free analogue of the Poisson limit theorem (see e.g. 
Proposition 12.11 in \cite{NS2006}).  The Marchenko-Pastur
distribution reviewed in Remark \ref{rem:64}(4) corresponds 
to $p=q=1$ (so ``$\Pi_1$'' from there becomes ``$\Pi_{1;1}$'').  
For general $p, q > 0$, the formula given in 
Remark \ref{rem:64}(4) for the $R$-transform of $\Pi_1$ extends 
to
\[
R_{ \Pi_{p;q} } (z) = \frac{pqz}{1 - qz}.
\]

Returning to the example of the $\lambda$-circular 
distribution, an immediate processing of the formulas
(\ref{eqn:62a}) from Theorem \ref{thm:62} gives us that the 
$R$-transforms of $ZZ^{*}$ and of $Z^{*}Z$ in the 
noncommutative probability space 
$( \bC \langle Z, Z^{*} \rangle, \nu)$ are
\[
R_{ZZ^{*}} (z) = \frac{\lambda z}{1-z}, \ \ 
R_{Z^{*}Z} (z) = \frac{z}{1- \lambda z}. 
\]
For our example, this shows that the distributions 
$\tau_1$ and $\tau_2$ appearing in  (\ref{eqn:66a})
 (Proposition \ref{prop:66}) are free Poisson distributions:
\[
\tau_1 = \Pi_{\lambda ; 1}  \mbox{ and }
\tau_2 = \Pi_{1 / \lambda ; \lambda}.
\]
\end{example}

$\ $

\section{Stability of $\Rcinfdiv$ under free 
multiplicative convolution}

\begin{remark}  \label{rem:71} 
In this section we consider the operation $\boxtimes$ on 
$\cD_c (1,*)$, which follows the {\em multiplication} of 
$*$-free random variables (cf. Definition \ref{def:41}(3)).  
One has the remarkable fact that whenever 
$\mu, \mu' \in \cD_c (1,*)$ and at least one of $\mu, \mu'$ 
is $R$-diagonal, it follows that $\mu \boxtimes \mu'$ is 
$R$-diagonal as well (see \cite[Proposition 15.8]{NS2006}).  
If we make the 
additional assumption that both $\mu$ and $\mu '$ are $R$-diagonal,
then we have explicit formulas for the determining sequences 
of $\mu \boxtimes \mu '$ in terms 
of the determining sequences of $\mu$ and of $\mu '$.  To be 
precise, denote the determining sequences of $\mu$
by $(\alpha_n)_{n=1}^\infty$, $(\beta_n)_{n=1}^\infty$, and 
those of $\mu ' $ by $(\alpha^\prime_n)_{n=1}^\infty$, 
$(\beta^\prime_n)_{n=1}^\infty$. Tthen the 
determining sequences 
$(\widehat{\alpha}_n)_{n=1}^\infty$, 
$(\widehat{\beta}_n)_{n=1}^\infty$ of $\mu \boxtimes \mu '$
are given by:
\begin{align}    \label{eqn:71a}
\begin{cases}
\widehat{\alpha}_n=\displaystyle\sum_{
	\substack{\pi \sqcup \rho\in NC(2n)\\
	\pi=\{V_1,\ldots,V_p\} \in NC(1,3,\ldots,2n-1)\\
         \text{with }1\in V_1 , \text{ and}                      \\
	\rho=\{W_1,\ldots,W_r \}\in NC(2,4,\ldots,2n)}}
	\alpha_{|V_1|}\beta_{|V_2|}\cdots\beta_{|V_p|} 
        \alpha^\prime_{|W_1|}\cdots\alpha^\prime_{|W_r|},
		\\ 
                                                                 \\
\widehat{\beta}_n=\displaystyle\sum_{
	\substack{\pi \sqcup \rho\in NC(2n)\\
	\pi=\{V_1,\ldots,V_p\} \in NC(1,3,\ldots,2n-1)\\
         \text{with }1\in V_1 , \text{ and}                      \\
	\rho=\{W_1,\ldots,W_r \}\in NC(2,4,\ldots,2n)}}
	\beta^\prime_{|V_1|}\alpha^\prime_{|V_2|}\cdots
        \alpha^\prime_{|V_p|}\beta_{|W_1|}\cdots\beta_{|W_r|}.
	\end{cases}
	\end{align}

The formulas (\ref{eqn:71a}) were proved in 
\cite[Proposition 3.9]{KS2000}.  They can also be rephrased in 
terms of equations for power series, as shown in the next 
proposition.  The formulas (\ref{eqn:72a}) in the
proposition have appeared before (but only as a conjecture, 
without proof), in \cite[Section 5.3]{NSS2001}.  For the 
reader's convenience, we include the proof of how (\ref{eqn:72a})
is derived out of (\ref{eqn:71a}).
\end{remark}

\begin{proposition} \label{prop:72}
With the notation of Remark \emph{\ref{rem:71}}, suppose that we have 
elements $a,b,a^\prime,b^\prime$ in a noncommutative probability 
space $(\cA,\varphi)$ such that $\{a,b\}$ is free from 
$\{a^\prime,b^\prime\}$ and such that 
\begin{align*}
R_a(z)    & =\sum_{n=1}^{\infty}\alpha_n z^n, \, 
            R_b(z)=\sum_{n=1}^{\infty}\beta_n z^n,  \\
R_{a'}(z) & =\sum_{n=1}^{\infty}\alpha^\prime_n z^n, 
     \, R_{b'}(z)=\sum_{n=1}^{\infty}\beta^\prime_n z^n.
\end{align*}	
Assume moreover that $\beta_1\ne 0\ne\alpha_1^\prime$, so 
the series $R_b$ and $R_{a'}$ have inverses $R_b^{\langle -1 \rangle}$ and
$R_{a'}^{\langle -1 \rangle}$ relative to composition.
Then:
\begin{align} \label{eqn:72a}
\begin{cases}
\sum_{n=1}^\infty \widehat{\alpha}_n z^n
  & =\left(R_a \circ R_b^{\langle -1 \rangle} 
               \circ M_{ba^\prime}\right) (z),
\\
\sum_{n=1}^\infty \widehat{\beta}_n z^n
  & =\left(R_{b^{\prime}} \circ R_{a^\prime}^{\langle -1 \rangle} 
                          \circ M_{a^\prime b}\right) (z).
\end{cases}
\end{align}
\end{proposition}

\begin{proof} The second equation in (\ref{eqn:72a}) follows from 
the first one if we substitute $b', a', b$ for $a,b,a'$, respectively. 
To prove the first equation, we fix an $n \in \bN$ and we suitably 
structure the formula for $\widehat{\alpha}_n$ provided in 
(\ref{eqn:71a}).  Let us also momentarily fix an $m\le n$ and a 
set $V_1=\{2i_1 -1,\dots,2i_m-1\}$, where $1=i_1<\cdots <i_m\le n$. 
Denote $n_k=i_{k+1}-i_k$, $k=1,\dots,m$, where $i_{m+1}=n+1$. Note 
that $n_1+\cdots+n_m=n$ and that $V_1$ can recovered from 
$n_1,\dots,n_m$. Use the moment-cumulant formula as in the proof 
of Proposition \ref{prop:36} (using the cumulant functionals and 
the fact that the mixed cumulants of $a'$ and $b$ vanish on account 
of freeness) to obtain
\[
\sum_{
	\substack{\pi \sqcup \rho\in NC(2n)\\
	\pi=\{V_1,\ldots,V_p\} \in NC(1,3,\ldots,2n-1)\\
	\rho=\{W_1,\ldots,W_r \}\in NC(2,4,\ldots,2n)}}
	\alpha_{|V_1|}\beta_{|V_2|}\cdots\beta_{|V_p|} 
        \alpha^\prime_{|W_1|}\cdots\alpha^\prime_{|W_r|}=\alpha_m\prod_{k=1}^m\varphi(a'(ba')^{n_k-1}).
\]
Letting $V_1$ vary,  (\ref{eqn:71a}) yields, for the $n \in \bN$ that 
we had fixed:
\[
\widehat{\alpha}_n=\sum_{m=1}^n\alpha_m\sum_{n_1+\cdots+n_m=n}
\prod_{k=1}^m\varphi(a'(ba')^{n_k-1}).
\]
We now let $n$ vary in $\bN$, and get that
\begin{equation}   \label{eqn:72b}
\sum_{n=1}^\infty \widehat{\alpha}_n z^n=\sum_{n=1}^\infty \alpha_m(g(z))^m=R_a(g(z)),
\end{equation}
where
\begin{equation}   \label{eqn:72c}
g(z) = \sum_{n=1}^{\infty} 
\varphi (  a' (ba')^{n-1}  ) z^n \in \bC [[z]].
\end{equation}

It remains to show that the series $g$ introduced in 
(\ref{eqn:72c}) is equal to 
$R_b^{\langle -1 \rangle} \circ M_{ba^\prime} $ or, equivalently, 
that one has $R_b\circ g= M_{ba^\prime} $.
To see this, apply again the moment-cumulant formula (using, as in \cite[Theorem 14.4]{NS2006} the fact that $b$ is free from $a'$) to obtain
\begin{equation}   \label{eqn:72d}
\varphi (  (ba')^n  ) 
= \displaystyle\sum_{
	\substack{\pi \sqcup \rho\in NC(2n)\\
	\pi=\{V_1,\ldots,V_p \} \in NC(1,3,\ldots,2n-1)\\
	\rho=\{W_1 ,\ldots,W_r \}\in NC(2,4,\ldots,2n)}}
	\beta_{|V_1|} \cdots \beta_{|V_p|} 
        \alpha^\prime_{|W_1|}\cdots \alpha^\prime_{|W_r|}.
\end{equation}
In (\ref{eqn:72d}) we can list the blocks of $\pi$ such that 
$1\in V_1$.  A similar argument to the one used above to structure
the formula for $\widehat\alpha_n$ shows now that
\[
\varphi (  (ba')^n  )=
\sum_{m=1}^n\beta_m\sum_{n_1+\cdots+n_m=n}\prod_{k=1}^m\varphi(a'(ba')^{n_k-1}),\quad n\in\mathbb N,
\]
and this implies the desired relation $M_{ba'}  = R_b \circ g$.
\end{proof}

\begin{remark}\label{remarca}
With the notation of the preceding proposition, suppose that $\beta_n=0$ for every $n\in\mathbb N$. Then the only non-zero term in the first equality in (\ref{eqn:71a}) corresponds to  $\pi=1_n$  and $\rho=0_n$, and therefore $\widehat\alpha_n=\alpha_n(\alpha'_1)^n$.
\end{remark}

The next corollary presents a reformulation of 
(\ref{eqn:72a}) which has the advantage that it introduces in 
discussion two power series $F$ and $\widetilde{F}$, related 
with the subordination results of \cite{B1998}.

\begin{corollary} \label{cor:73}
In the framework of Proposition \emph{\ref{prop:72}}, we have
	\begin{align} \label{eqn:73a}
	\begin{cases}
	\sum_{n=1}^\infty \widehat{\alpha}_n z^n&=R_a \left( F(z) \left(1+M_b(F(z))\right) \right)
	\\ 
	\sum_{n=1}^\infty \widehat{\beta}_n z^n&=R_{b^\prime} \left( \widetilde{F}(z) \left(1+M_{a^\prime}(\widetilde{F}(z))\right) \right),
	\end{cases}
	\end{align}
where $F=M_b^{\langle -1 \rangle} \circ M_{ba^\prime}$ and 
$\widetilde{F}=M_{a^\prime}^{\langle -1 \rangle} 
\circ M_{a^\prime b}$.
\end{corollary}

\begin{proof}  By symmetry, it suffices to prove the first of the two equations. Using (\ref{eqn:72a}), we see that  we must verify the identity
\begin{equation}   \label{eqn:73b}
( R_b^{\langle -1 \rangle} \circ M_{ba'} ) (z)
= F(z) \bigl( 1 + M_b(F(z)) ).
\end{equation}
Recalling the assumption that $\varphi (b) \neq 0$, the functional equation
$M_b(z)=R_b(z(1+M_b(z)))$ can be rewritten as
\begin{equation}    \label{eqn:73c}
R_b^{\langle -1 \rangle} (w) 
= (1+w) M_b ^{\langle -1 \rangle} (w)
\end{equation}
(see  \cite[Remark 16.18]{NS2006}).
Substitute $M_{ba'}$ for $w$ in (\ref{eqn:73c}) to find that 
\[
\bigl( R_b^{\langle -1 \rangle} \circ M_{ba'} \bigr) (z) 
= (1+ M_{ba'} (z) ) \cdot 
\bigl( M_b ^{\langle -1 \rangle} \circ M_{ba'} \bigr) (z)
= (1+ M_{ba'} (z) ) \cdot F(z).
\]
Finally, the definition of $F$ implies that 
$M_{ba'} (z) = M_b (F(z))$, and using this equality in the right hand side of the preceding equality yields (\ref{eqn:73b}).
\end{proof}

The following lemma is an immediate consequence of the 
definition of $\Psi$ (Remark \ref{rem:61}). 

\begin{lemma}   \label{lemma:77}
Consider a distribution $\nu \in \Rcinfdiv$, and let 
$( \alpha_n )_{n=1}^{\infty}$ and 
$( \beta_n )_{n=1}^{\infty}$ be its determining sequences.
There exist positive elements $a,b$ in a 
$C^{*}$-probability space $( \cA , \varphi )$ such that 
\[
R_a (z) = \sum_{n=1}^{\infty} \alpha_n z^n, 
\quad\text{and}\quad R_b (z) = \sum_{n=1}^{\infty} \beta_n z^n.
\]
Moreover, the distributions of $a$ and $b$ are 
$\boxplus$-infinitely divisible.
\end{lemma}

\begin{proof}  Write $\nu = \Psi ( \sigma_1, \sigma_2 )$, 
with $\sigma_1, \sigma_2 \in \cPplus_c$.  Then 
\[
\sum_{n=1}^{\infty} \alpha_n z^n = \eta_{\sigma_1} (z)
= R_{ \bB ( \sigma_1 ) } (z),
\]
where $\bB ( \sigma_1 )$ is a $\boxplus$-infinitely divisible 
distribution in $\cPplus_c$ (cf. Remark \ref{rem:64}(2)).
Thus taking $a$ to be a positive element with distribution 
$\bB ( \sigma_1 )$ in some $C^{*}$-probability space will 
fulfill the required conditions.  The argument for $b$ is 
similar.
\end{proof}

In reference to the set of power series $\cEplus_c$ introduced 
in Notation \ref{def:62a}, we record a result 
which follows easily from \cite[Proposition 2.2]{BB2005}.

\begin{proposition}   \label{prop:76}
A series $f \in \bC [[z]]$ belongs to the set $\cEplus_c$
 if and only if it satisfies
the following three conditions:
\begin{enumerate}
\item[\rm(i)] $f$ has real coefficients;
\item[\rm(ii)] $f$ has positive convergence radius;
\item[\rm(iii)] $f$ can be extended to an analytic map \emph{(}still denoted $f$\emph{)} of 
$\bC^{+}$ into $\overline{\bC^{+}}$ such that $f(0)=0$ and
$\mathrm{Arg}(z) \leq \mathrm{Arg}(f(z))$ for $z \in \bC^{+}$.
\end{enumerate}
\hfill $\square$
\end{proposition}

\begin{corollary}\label{corolar}
Suppose that the series 
$f(z)=\sum_{n=1}^\infty\alpha_n z^n\in\cEplus_c$ 
is not identically zero.  Then $\alpha_1>0$.
\end{corollary}

\begin{proof}
Let $n$ be the smallest integer such that $\alpha_n\ne0$ and suppose, 
to get a contradiction, that either $n>1$ or $n=1$ and $\alpha_n<0$.  Choose  $\gamma\in\mathbb{C}^+$ 
such that $|\gamma|=1$ and $\Im(\gamma^n\alpha_n)<0$. 
We have $\lim_{r\downarrow0}(f(r\gamma)/r^n))=\gamma^n\alpha_n$, 
and therefore $\Im f(r\gamma)<0$ for sufficiently small $r$, 
contrary to Proposition \ref{prop:76}(iii).
\end{proof}

We are now ready for the main result of this section.

\begin{theorem}  \label{thm:78}
For every  $\nu,\nu'\in\Rcinfdiv$ we have  $\nu \boxtimes \nu'\in\Rcinfdiv$.
\end{theorem}

\begin{proof}  Let
$( \alpha_n )_{n=1}^{\infty}$ and 
$( \beta_n )_{n=1}^{\infty}$ (respectively $( \alpha_n ')_{n=1}^{\infty}$ and 
$( \beta_n ' )_{n=1}^{\infty}$) denote the determining sequences 
of $\nu$ (respectively, $\nu'$). Two applications of Lemma 
\ref{lemma:77}, combined with a free product 
construction, allow us to construct a
$C^{*}$-probability space $( \cA , \varphi )$ and 
positive elements $a,b,a',b' \in \cA$ such that 
\begin{enumerate}
\item[(a)] $R_a (z) = \sum_{n=1}^{\infty} \alpha_n z^n$ and $R_b (z) = \sum_{n=1}^{\infty} \beta_n z^n$,
\item[(b)] $R_{a'} (z) = \sum_{n=1}^{\infty} \alpha_n' z^n$ and $R_{b'} (z) = \sum_{n=1}^{\infty} \beta_n' z^n$, and
\item[(c)]$\{ a,b \}$ is free from $\{ a',b' \}$.
\end{enumerate}
We know from Remark \ref{rem:71} 
 that $\nu\boxtimes\nu^\prime$ is an $R$-diagonal 
distribution in $\Dc (1,*)$. Let  $(\widehat{\alpha}_n)_{n=1}^{\infty}$ and 
$(\widehat{\beta}_n)_{n=1}^{\infty}$  denote the determining sequences of $\nu\boxtimes\nu^\prime$, and set
\[
\widehat{f} (z) := \sum_{n=1}^{\infty} 
\widehat{\alpha}_n z^n, \quad 
\widehat{g} (z) := \sum_{n=1}^{\infty} \widehat{\beta}_n z^n. 
\]
By Proposition \ref{prop:62b}, we have to prove that 
$\widehat{f},\widehat{g}\in\cEplus_c$. By symmetry, it suffices to 
show that $\widehat{f}\in\cEplus_c$, and this is done by verifying 
that $\widehat f$ satisfies conditions (i)--(iii) of Proposition 
\ref{prop:76}.  We dispose first of the simple case in which $\beta_1=0$. 
Corollary \ref{corolar} yields $\beta_n=0$ for all $n\in\mathbb N$, and 
Remark \ref{remarca} implies that
$\widehat f(z)=\sum_{n=1}^\infty \alpha_n (\alpha'_1 z)^n$. The desired conclusion follows because $\alpha'_1\ge 0$ and 
the series 
$\sum_{n=1}^\infty \alpha_n z^n$ belongs to $\cEplus_c$. Similarly, 
if $\alpha'_1=0$, Corollary \ref{corolar} yields $\alpha'_n=0$ for 
all $n\in\mathbb N$, and then (\ref{eqn:71a}) implies that $\widehat f=0$.

It remains to show that $\widehat f\in\cEplus_c$ when $\beta_1 \neq 0 \neq \alpha_1 '$. In this case, Corollary \ref{cor:73} shows that $\widehat f=R_a \left( F(z) \left(1+M_b(F(z))\right) \right)$, where $F = M_b^{\langle -1 \rangle} \circ M_{ba'}$. In other words, $\widehat f$ is the composition of the three power series $R_a(z)$, $z(1+M_b(z))$, and $F(z)$. We know that $R_a\in\cEplus_c$. 
It was proved in \cite{B1998} that $F\in\cEplus_c$. The series 
$z(1+M_b(z))$ also belongs to $\cEplus_c$ by 
\cite[Proposition 6.1]{BV1993}. Proposition \ref{prop:76} shows 
that the set $\cEplus_c$ is closed under composition. Therefore $\widehat f\in\cEplus_c$, thus concluding the proof.
\end{proof}

\begin{corollary}\label{cor:79}
Suppose that $\nu\in\Rcinfdiv$ is the $*$-distribution of an 
element $a$ in 
some $C^*$-probability space. Then the $*$-distribution of $a^n$ 
belongs to $\Rcinfdiv$ for every $n\in\mathbb N$.
\end{corollary}

\begin{proof}  This follows from Theorem \ref{thm:78}
and the known fact   \cite[Proposition 3.11]{KS2000} that the 
distribution of $a^n$ is equal to $\nu^{\boxtimes n}$.
\end{proof}

We conclude the section by looking again at the KMS example,
and by describing explicitly the BBP parametrization for the 
powers of a $\lambda$-circular element.

\begin{remark}\label{rem:710}

(1) Suppose that $\nu, \nu ' \in \Dc (1,*)$ are $R$-diagonal and  
satisfy the KMS condition with parameters $t,t'\in( 0, \infty )$, 
respectively.  Consider the $*$-distribution $\nu \boxtimes \nu '$,
which is $R$-diagonal as well (see Remark \ref{rem:71}).  We 
claim that $\nu \boxtimes \nu '$ also satisfies the KMS condition, 
with parameter $t  t'$.  Using the same notations for determining 
sequences as in Remark \ref{rem:71}, this claim amounts to the 
fact that $\widehat{\alpha}_n=(t t^\prime) \widehat{\beta}_n$
for every $n\in\mathbb N$.  In order to prove this, we replace 
$\alpha_{|V_1|}$ and $\alpha'_{|W_1|}$ by $t\beta_{|V_1|}$ and 
$t\beta'_{|W_1|}$, respectively, in the first formula 
\eqref{eqn:71a} to obtain
\[
		\widehat{\alpha}_n= (t t^\prime) 	\sum_{
		\substack{\pi \sqcup \rho\in NC(2n)\\
				\pi=\{V_1,\ldots,V_k\} \in NC(1,3,\ldots,2n-1)\\
				\rho=\{W_1,\ldots,W_l\}\in NC(2,4,\ldots,2n)\\ 1\in V_1,2\in W_1}}
		\beta^\prime_{|W_1|}\alpha^\prime_{|W_2|}\cdots\alpha^\prime_{|W_l|} \beta_{|V_1|}\beta_{|V_2|}\cdots\beta_{|V_k|}.
\]
To see that the last sum equals $\widehat{\beta}_n$, we observe 
that pairs $(\pi,\rho)$ as above are in a bijective correspondence 
with pairs $(\widetilde{\pi},\widetilde{\rho})$ such that 
$\widetilde{\pi} \sqcup \widetilde{\rho}\in NC(2n)$ and 
$\widetilde{\pi} = \{\widetilde{W}_1,\ldots,\widetilde{W}_l\}
\in NC(1,3,\ldots,2n-1)$ and 
$\widetilde{\rho} = \{\widetilde{V}_1,\ldots,\widetilde{V}_k\} 
\in NC(2,4,\ldots,2n)$. 
Indeed,  $\widetilde{\rho}$ and $\widetilde{\pi}$ are obtained as $\widetilde{\pi}\sqcup\widetilde{\rho}=\gamma_{2n}^{-1}(\pi\sqcup\rho)$, where we use the permutation $\gamma_{2n}$ from the proof of Proposition \ref{prop:37}. Thus, the sum above is equal to\[
		\sum_{
			\substack{\pi^\prime \sqcup \rho^\prime\in NC(2n)\\
				\widetilde{\pi}=\{\widetilde{W}_1,\ldots,\widetilde{W}_l\} \in NC(1,3,\ldots,2n-1)\\
				\widetilde{\rho}=\{\widetilde{V}_1,\ldots,\widetilde{V}_k\}\in NC(2,4,\ldots,2n)\\ 1\in\widetilde{W}_1,2\in\widetilde V_1}}
		\beta^\prime_{|\widetilde{W}_1|}\alpha^\prime_{|\widetilde{W}_2|}\cdots\alpha^\prime_{|\widetilde{W}_l|} \beta_{|\widetilde{V}_1|}\cdots\beta_{|\widetilde{V}_k|},
\]
and this equals $\widehat\beta_n$ by the second formula  
\eqref{eqn:71a}.

\bigskip

(2) Now fix a real number $\lambda>0$ and  consider 
the $\lambda$-circular distribution 
$\nu= \Psi(\delta_\lambda,\delta_1)$, as in Example \ref{ex:68}.
If $a$ is an element in some $*$-probability space such that 
the $*$-distribution of $a$ is equal to $\nu$, then we will say 
that $a$ is a {\em $\lambda$-circular element}.  Such elements 
do of course exist, for instance we can just take $a = Z$ in the 
$*$-probability space $( \bC \langle Z, Z^{*} \rangle, \nu )$.
If $a$ is a $\lambda$-circular element, then Theorem \ref{thm:78} 
and Corollary \ref{cor:79} tell us that every power $a^k$ has 
$*$-distribution $\nu^{\boxtimes k} \in \Rcinfdiv$.  Moreover, 
part (1) of the present remark assures us that 
$\nu^{\boxtimes k}$ satisfies the KMS condition with parameter 
$\lambda^k$.  Hence for every $k \in \bN$ we have a BBP 
parametrization of the form 
\[
\nu^{\boxtimes k} = \Psi(\sigma_k^{\uplus\lambda^k},\sigma_k), 
\]
for some probability measure $\sigma_k \in \cPplus_c$.  For 
$k=1$, we know from Example \ref{ex:68} that $\sigma_1$ is 
the Dirac mass $\delta_1$.  The next proposition gives a way
of describing  $\sigma_k$ for $k \geq 2$.
\end{remark}

\begin{proposition}   \label{prop:713}
Let $\lambda$ and $( \sigma_k )_{k=1}^{\infty}$ be as above,
and consider on the other hand the probability measures with
finite support $( \tau_k )_{k=1}^{\infty}$ defined by
\begin{align*}
\tau_k := \frac{\lambda^k}{1+\lambda^k} \, \delta_0
+\frac{1}{1+\lambda^{k}} \, \delta_{1 + \lambda^{k}},
\ \ k \in \bN.
\end{align*}
Then one has
\begin{equation}   \label{eqn:713g}
\sigma_k = \tau_1 \boxtimes \cdots \boxtimes \tau_{k-1} , 
\ \ k \geq 2.
\end{equation}
\end{proposition}

\begin{proof}  As in Remark \ref{rem:710}(2), we use the 
notation $\nu$ for the $\lambda$-circular distribution.  We fix
a $k \in \bN$ and  invoke Proposition \ref{prop:72} in the 
special case in which the $*$-distributions $\mu, \mu '$ considered 
there are $\nu^{\boxtimes k}$ and $\nu$, respectively.  The 
power series
\begin{equation}   \label{eqn:713a}
\sum_{n=1}^{\infty} \beta_n z^n, \ \ 
\sum_{n=1}^{\infty} \beta_n ' z^n, \ \ 
\sum_{n=1}^{\infty} \widehat{\beta}_n z^n
\end{equation}
from Proposition \ref{prop:72} are equal in this case to the 
$\eta$-series of the probability measures $\sigma_k$,
$\sigma_1$ and $\sigma_{k+1}$, respectively.  (For instance the 
equality $\sum_{n=1}^{\infty} \beta_n z^n = \eta_{\sigma_k} (z)$ 
follows from the comments at the end of Remark \ref{rem:61} and 
the fact that $\nu^{\boxtimes k} 
= \Psi ( \sigma_k^{\uplus \lambda^k}, \sigma_k )$.)
Note that, since $\sigma_1 = \delta_1$,
for the second power series in (\ref{eqn:713a}) we actually 
have $\sum_{n=1}^{\infty} \beta_n ' z^n = z$.

The notation of Proposition \ref{prop:72} also include some 
non-commutating random variables $a,b,a',b'$, where
$b$ is such that
\begin{equation}   \label{eqn:713b}
R_b (z) = \sum_{n=1}^{\infty} \beta_n z^n 
= \eta_{\sigma_k} (z)
= R_{ \bB ( \sigma_k ) } (z).
\end{equation}
From (\ref{eqn:713b}) we infer that the distribution of $b$ is 
$\bB ( \sigma_k )$.  Similar reasoning, based on the 
formulas $R_{b'} (z) = z$ and $R_{a'} (z) = \lambda z$, leads
to the fact that $a'$ and $b'$ have distributions 
$\delta_{\lambda}$ and $\delta_1$, respectively.  As a consequence,
we may assume without loss of generality that $a' = \lambda$
and $b' = 1$ in their noncommutative probability space.

We are interested in the second relation (\ref{eqn:72a}) from
Proposition \ref{prop:72}.  Due to the very simple form
of $R_{a'}$ and $R_{b'}$, this equation simplifies to 
\begin{equation}   \label{eqn:713c}
\sum_{n=1}^{\infty} \widehat{\beta}_n z^n
= \frac{1}{\lambda} M_{\lambda b} (z).
\end{equation}
The same argument as used in (\ref{eqn:713b}) shows that 
the left-hand side of (\ref{eqn:713c}) is equal to 
$R_{ \bB ( \sigma_{k+1} ) } (z)$.  On the right-hand side of 
(\ref{eqn:713c}) we perform the obvious transformation 
$M_{\lambda b} (z) = M_b ( \lambda z )
= M_{ \bB ( \sigma_k ) } ( \lambda z)$, and this leads us to 
a direct connection between 
$\sigma_k$ and $\sigma_{k+1}$:
\begin{equation}   \label{eqn:713d}
R_{ \bB ( \sigma_{k+1} ) } (z)
= \frac{1}{\lambda} 
M_{ \bB ( \sigma_k ) } ( \lambda z). 
\end{equation}

In order to make use of (\ref{eqn:713d}), it is convenient to 
resort to another well-known transform of free probability, 
the $S$-transform.  For a probability measure $\sigma \in \cP_c$ 
with non-vanishing mean, one defines the $S$-transform of 
$\sigma$ as the power series
\[
S_{\sigma} (z) = \frac{1}{z} R_{\sigma}^{\langle -1 \rangle} (z)
= \frac{z+1}{z} M_{\sigma}^{\langle -1 \rangle} (z)
\]
(see, for instance, 
\cite[Definition 18.15 and Remark 18.16 on p. 294]{NS2006}).
Some straightforward processing of Equation (\ref{eqn:713d}) 
(multiply both sides by $\lambda$, take inverses under 
composition, and write the resulting series in terms of the 
suitable $S$-transforms) 
then leads to the formula
\begin{equation}   \label{eqn:713e}
S_{ \bB ( \sigma_{k+1} ) } (z)
= \frac{1}{1 + \lambda z} 
S_{ \bB ( \sigma_k ) } ( \lambda z). 
\end{equation}

The formula (\ref{eqn:713e}) was obtained for a fixed (but 
arbitrary) $k \in \bN$.  We now unfix $k$ and use a 
straightforward induction argument, with base case 
$S_{ \bB ( \sigma_1 ) } (z) = S_{\delta_1} (z) = 1$, in 
order to infer that 
\begin{equation}   \label{eqn:713f}
S_{ \bB ( \sigma_{k} ) } (z)
= \prod_{j=1}^{k-1} \frac{1}{1 + \lambda^j z}, 
\ \ \forall \, k \in \bN .
\end{equation}

It remains to make the connection to the $\tau_k$
indicated in the statement of the proposition.  For every 
$j \in \bN$, an elementary calculation shows that 
$\bB( \tau_j)$ is the free Poisson distribution 
$\Pi_{1/\lambda^{j};\lambda^j}$, where 
the notation ``$\Pi_{p;q}$'' is as in Example $\ref{ex:68}$.
Another elementary calculation shows that the 
$S$-transform of $\Pi_{1/\lambda^{j};\lambda^j}$ is 
$1/( 1+ \lambda^j z)$.  Thus the right-hand side of 
(\ref{eqn:713f}) can be written as 
$S_{ \bB ( \tau_1 ) } (z) \, S_{ \bB ( \tau_2 ) } (z)
\cdots S_{ \bB ( \tau_{k-1} ) } (z)$.

Now, the $S$-transform is multiplicative with respect to the 
operation $\boxtimes$ (\cite[Corollary 18.17]{NS2006}).  
Since $\bB$ is multiplicative as well (Remark \ref{rem:64}(3)), 
the observations made in the preceding paragraph lead to the 
formula
\[
S_{ \bB ( \sigma_{k} ) } 
= S_{ \bB ( \tau_1 \boxtimes \cdots \boxtimes \tau_{k-1} ) },
\ \ k \geq 2.
\]
The required Equation (\ref{eqn:713g}) follows from here, 
since $\bB$ is injective and since a probability measure with 
non-vanishing mean is uniquely determined by its $S$-transform. 
\end{proof}

\end{document}